\pdfoutput=1
\documentclass[10pt,a4paper,reqno]{amsart}
\usepackage{tabls}
\usepackage{url}
\usepackage[linktocpage = true]{hyperref}
\hypersetup{
 colorlinks = true,
 citecolor = blue,
}
\usepackage{color}
\definecolor{red}{rgb}{1,0,0}
\definecolor{green}{rgb}{0,1,0}
\definecolor{blue}{rgb}{0,0,1}
\definecolor{refkey}{gray}{.625}
\definecolor{labelkey}{gray}{.625}
\usepackage[lite]{amsrefs}
\usepackage[a4paper]{geometry}
\usepackage{amsfonts}
\usepackage{amssymb}
\usepackage{mathrsfs}
\usepackage{amsmath}
\usepackage{amsthm}
\usepackage{stmaryrd}
\usepackage{amscd}
\usepackage{enumerate}
\usepackage{paralist}
\usepackage{graphicx}
\usepackage{mathtools}
\usepackage[all,cmtip]{xy}
\usepackage{kantlipsum}
\usepackage{tikz-cd}
\allowdisplaybreaks

\newcommand{\abs}[1]{\lvert#1\rvert}

\newcommand{\Ad}{\operatorname{Ad}}

\newcommand{\g}{\mathfrak{g}}
\newcommand{\h}{\mathfrak{h}}

\makeatletter
 \def\title@font{\normalsize\bfseries}
 \let\ltx@maketitle\@maketitle
 \def\@maketitle{\bgroup%
 \let\ltx@title\@title%
 \def\@\title{\resizebox{\textwidth}{!}{%
  \mbox{\title@font\ltx@title}%
 }}%
 \ltx@maketitle%
 \egroup}
\makeatother

\theoremstyle{plain}
\newtheorem{thm}[equation]{Theorem}
\newtheorem{lem}[equation]{Lemma}
\newtheorem{Cor}[equation]{Corollary}
\newtheorem*{theorem*}{Theorem}
\theoremstyle{definition}
\newtheorem{defn}[equation]{Definition}
\newtheorem{example}[equation]{Example}
\newtheorem{prop}[equation]{Proposition}

\newtheorem{Rem}[equation]{Remark}
\numberwithin{equation}{section}

\begin{document}
\def\C{\mathbb{C}}
\def\CE{\mathrm{CE}}
\def\LP{\mathrm{LP}}
\def\E{\mathscr{E}}
\def\k{\mathbb{K}}
\def\Im{\text{Im}}
\def\End{\operatorname{End}}
\def\pr{\operatorname{pr}}
\def\Lie{\operatorname{Lie}}
\def\id{\operatorname{id}}
\def\Der{\operatorname{Der}}
\def\Hom{\operatorname{Hom}}
\def\Rep{\operatorname{Rep}}
\def\Map{\operatorname{Map}}
\def\Mod{\operatorname{Mod}}
\def\sgn{\operatorname{sgn}}

\title{Cohomology of hemistrict Lie 2-algebras}

\author{Xiongwei Cai}
\address{Chern Institute of Mathematics, Nankai University, Tianjin}
\email{\href{mailto:shernvey@nankai.edu.cn}{shernvey@nankai.edu.cn}}

\author{Zhangju Liu}
\address{Department of Mathematics, Peking University, Beijing}
\email{\href{mailto:liuzj@pku.edu.cn}{liuzj@pku.edu.cn}}

\author{Maosong Xiang}
\address{Center for Mathematical Sciences, Huazhong University of Science and Technology, Wuhan}
\email{\href{mailto: msxiang@hust.edu.cn}{msxiang@hust.edu.cn}}
\thanks{Research partially supported by NSFC grants 11425104, 11901221 and 11931009.}

\begin{abstract}
  We study representations of hemistrict Lie 2-algebras and give a functorial construction of their cohomology. We prove that both the cohomology of an injective hemistrict Lie 2-algebra $L$ and the cohomology of the semistrict Lie 2-algebra obtained from skew-symmetrization of $L$ are isomorphic to the Chevalley-Eilenberg cohomology of the induced Lie algebra $L_{\Lie}$. 
\end{abstract}

\maketitle

\tableofcontents

\section{Introduction}
The notion of weak Lie 2-algebras was introduced by Roytenberg in~\cite{Roy1} to complete the picture of categorification of Lie algebras started by Baez and Crans~\cite{BC}. Roughly speaking, a weak Lie 2-algebra is a bilinear bracket on a linear category such that both skew-symmetry and Jacobi identity hold only up to natural transformations, called alternator and Jacobiator, respectively.  It was further shown in \emph{op.cit.} that the 2-category of weak Lie 2-algebras is equivalent to the 2-category of 2-term weak $L_\infty$-algebras by passing to the normalized chain complex. By a weak $L_\infty$-algebra, we mean a Loday infinity algebra~\cite{AP} $(V,\{\pi_k\}_{k\geq1})$ whose structure maps $\pi_k$ are skew-symmetric up to homotopy. Note that the cohomology of the underlying Loday infinity algebra of a weak $L_\infty$-algebra defined in \emph{loc.cit.} by Ammar and Poncin cannot encode the additional information on the weak symmetry of the structure maps.
This is the first of a series of papers devoted to the study of a new cohomology theory of weak $L_\infty$-algebras and its relation with other known cohomology theories.

The purpose of this paper is to study cohomology of hemistrict Lie 2-algebras. Such an algebraic structure is specified by a bilinear bracket $[-,-]$ on a 2-term cochain complex (or a 2-vector space) $\vec{L} = L^{-1} \xrightarrow{d} L^0$, which is skew-symmetric up to a chain homotopy $h_2$, called alternator, and satisfies the Jacobi (or Leibniz) identity  (see Definition~\ref{Def of weak Lie 2-algebra}). Thus a hemistrict Lie 2-algebra is indeed a 2-term differential graded (dg for short) Leibniz algebra whose bracket is skew-symmetric up to homotopy.
As an immediate example, each Leibniz algebra $(\g, [-,-]_\g)$ gives rise to a hemistrict Lie 2-algebra
\[
  L_\g = (K[1] \hookrightarrow \g, [-,-]_\g, h_2),
\]
where $K$ is the Leibniz kernel, and $h_2$ is the composition of the degree shifting operator $[1]$ and the $K$-valued symmetric pairing $h$ on $\g$ defined by $h(x,y) = [x,y]_\g + [y,x]_\g$ for all $x,y \in \g$.

We first study representations of a hemistrict Lie 2-algebra $L = (\vec{L},[-,-],h_2)$ on a 2-term cochain complex $\vec{V}$ in Section~\ref{Sec: rep}. By forgetting the alternator $h_2$, it follows that each representation of $L$ gives rise to a representation of the dg Leibniz algebra $(\vec{L},[-,-])$. Furthermore, the space of representations of a hemistrict Lie 2-algebra $L$ on $\vec{V}$ is one-to-one correspondent to the space of semidirect products of hemistrict Lie 2-algebras of $L$ by $\vec{V}$ (see Proposition~\ref{prop: equiv de module}).

We then focus on a construction of cohomology of hemistrict Lie 2-algebras in Section~\ref{sec: coh}. Our approach originates from Roytenberg's construction of standard complexes for Courant-Dorfman algebras in~\cite{Roy2}: Recall that a Courant-Dorfman algebra is quintuple $(\mathcal{R},\mathcal{E},\langle-,-\rangle,\partial, [-,-])$, where $\mathcal{R}$ is a commutative algebra, $(\mathcal{E},\langle-,-\rangle)$ is a metric $\mathcal{R}$-module,
$\partial$ is an $\mathcal{E}$-valued derivation of $\mathcal{R}$, and $[-,-]$ is a Dorfman bracket on $\E$. All the data subject to several compatible conditions generalizing those defining a Courant algebroid~\cite{LWX}.
Denote by $\Omega^1$ the K\"{a}hler differential of $\mathcal{R}$.
Roytenberg associates to each metric $\mathcal{R}$-module $(\mathcal{E},\langle-,-\rangle)$ a graded commutative subalgebra $C(\mathcal{E},\mathcal{R})$ of the convolution algebra $\Hom(U(\mathcal{L}),\mathcal{R})$, where $\mathcal{L} = \mathcal{E}[1] \oplus \Omega^1[2]$ is the graded Lie algebra whose bracket is given by the composition $d_{dR}\langle-,-\rangle$ of the $\mathcal{R}$-valued metric $\langle-,-\rangle$ and the universal de Rham differential $d_{dR}: \mathcal{R} \to \Omega^1$. Moreover, the derivation $\partial$ and the Dorfman bracket $[-,-]$ induces a natural differential $D$ on $C(\mathcal{E},\mathcal{R})$. The resulting complex was called in \emph{op.cit.} the standard complex of this Courant-Dorfman algebra.
Note that the prescribed differential $D$ is indeed defined on the whole convolution algebra $\Hom(U(\mathcal{L}),\mathcal{R})$. This larger cochain complex $(\Hom(U(\mathcal{L}),\mathcal{R}),D)$ includes the information of the homotopy term of the underlying hemistrict Lie 2-algebra $(\Omega^1[1] \to \mathcal{E},\{-,-\},d_{dR}\langle-,-\rangle)$ (see Example~\ref{Ex: CD-algebra}), thus to some extent, encodes a new cohomology in need.

We associate to each representation $V$ of a hemistrict Lie 2-algebra $L = (\vec{L},[-,-],h_2)$ a cochain complex $C^\bullet(L,V)$, also called the standard complex, where the pair $(\vec{L},h_2)$ plays a similar role as the metric $\mathcal{R}$-module in Roytenberg's construction.
The differential is the restriction of the Loday-Pirashvili differential $D$ of the dg Leibniz algebra $(\vec{L},[-,-])$ (see Lemma~\ref{Lemma: differential}). The cohomology of the representation $V$ is defined to be the cohomology of $C^\bullet(L,V)$ (see Definition~\ref{Def: coh}).
Applying this construction of standard complex to Leibniz algebras, it is shown in~\cite{CaiL} that the Leibniz bracket of a fat Leibniz algebra can be realized as a derived bracket.

Note that the construction of standard complexes depends on both a hemistrict Lie 2-algebra and a representation. It is natural to ask how it varies with respect to these two objects. First of all, when we fix a hemistrict Lie 2-algebra $L$, the construction of standard complexes is natural with respect to representations (See Proposition~\ref{prop: functoriality1}).
Meanwhile, the construction of standard complexes is also functorial with respect to the morphisms of hemistrict Lie 2-algebras (see Theorem~\ref{Thm: Functoriality}).

In Section~\ref{Sec: app}, we study the cohomology of hemistrict Lie 2-algebras of a particular type.
A hemistrict Lie 2-algebra $L = (\vec{L},[-,-],h_2)$ is said to be injective if the differential $d$ of the 2-term cochain complex $\vec{L}$ is injective, i.e., $H^\bullet(\vec{L}) = H^0(\vec{L}) = L^0/dL^{-1}$. The bracket $[-,-]$ induces a Lie algebra structure on $H^0(\vec{L})$. This Lie algebra is denoted by $L_{\Lie}$. Meanwhile, according to Roytenberg~\cite{Roy1}, there is a semistrict Lie 2-algebra $\tilde{L}$ obtained from skew-symmetrization.
We prove the following
\begin{theorem*}[see Theorem~\ref{main theorem}]
  Let $L$ be an injective hemistrict Lie 2-algebra with a representation $V$ such that $l_\alpha = 0$ for all $\alpha \in L^{-1}$. Then both the Lie algebra $L_{\Lie}$ and the semistrict Lie 2-algebra $\tilde{L}$ admit a natural representation on $V$. Moreover,
  \[
   H^\bullet(L,V) \cong H^\bullet(\tilde{L},V) \cong H^\bullet_{\CE}(L_{\Lie},V).
  \]
\end{theorem*}
As an application, let $\g$ be a Leibniz algebra and $L_\g$ the associated hemistrict Lie 2-algebra. According to Roytenberg~\cite{Roy1}, and independently Sheng and Liu~\cite{SL}, the skew-symmetrization of the Leibniz bracket $[-,-]_\g$ induces a semistrict Lie 2-algebra
\[
 \mathcal{G} = (K[1] \hookrightarrow \g, \widetilde{l}_2,\widetilde{l}_3).
\]
As a consequence, we have
\begin{theorem*}[see Theorem~\ref{Thm: Leibniz algebra}]
  Let $\g$ be a Leibniz algebra with Leibniz kernel $K$ and $V$ a representation of the associated hemistrict Lie 2-algebra $L_\g$ to $\g$ such that $l_\alpha = 0$ for all $\alpha \in K$. Then
  \[
    H^\bullet(L_\g,V) \cong H^\bullet(\mathcal{G},V) \cong H^\bullet_{\CE}(\g_{\Lie},V).
  \]
\end{theorem*}

\paragraph{\bf The sequel(s)}
We plan to write two sequels to this paper, in which we address several issues not covered here: In the first one in preparation, we study cohomology of weak Lie 2-algebras, which encodes cohomology of hemistrict Lie 2-algebras and semistrict Lie 2-algebras into a unified framework. Our first goal is to establish the compatibility between the functor of taking cohomology and the skew-symmetrization functor~\cite{Roy1} from the category of weak Lie 2-algebras to the category of semistrict Lie 2-algebras.
Meanwhile, in their work~\cite{CSX} of weak Lie 2-bialgebras, Chen, Sti\'{e}non and Xu developed an odd version of big derived bracket approach to semistrict Lie 2-algebras. Our second goal is to establish a derived bracket formalism for cohomology of weak Lie 2-algebras.

In the second one, we consider weak $L_\infty$-algebroids which encodes a Courant algebroid as a 2-term weak $L_\infty$-algebroid.
According to Kontsevich and Soibelman~\cite{KS}, $L_\infty$-algebras correspond to formal pointed dg manifolds, while $A_\infty$-algebras correspond to noncommutative formal pointed dg manifolds. Our purposes are to reinterpret weak $L_\infty$-algebroids as certain commutative up to homotopy formal dg manifolds and to investigate their relation with shifted derived Poisson manifolds studied in~\cite{BCSX} and strongly homotopy Leibniz algebra over a commutative dg algebra studied in~\cite{CLX}.

\paragraph{\bf Acknowledgement}
We would like to express our gratitude to several institutions for their hospitality while we were working on this project: Chern Institute of Mathematics (Xiang), Henan Normal University (Cai and Xiang), Peking University (Cai and Xiang), Tsinghua University (Cai and Xiang).
We would also like to thank Chengming Bai, Zhuo Chen, Yunhe Sheng, Rong Tang, and Tao Zhang for helpful discussions and comments.
Special thanks go to Zhuo Chen for constructive suggestions on this manuscript.
We are grateful to the anonymous referee for carefully reading this paper and providing us valuable suggestions.

\section{Representations of hemistrict Lie 2-algebras}\label{Sec: rep}
In this section, we study representations of hemistrict Lie 2-algebras on 2-term cochain complexes as Beck modules of the category of hemistrict Lie 2-algebras.

\subsection{Dg Leibniz algebras and hemistrict Lie 2-algebras}
As the first step, we collect some basic facts of dg Leibniz algebras, which are Leibniz algebras (or Loday algebras) introduced by Loday~\cite{Loday} in the category of cochain complexes (see~\cite{DMS} for more structure theorems on Leibniz algebras). Denote by $\k$ the base field which is either $\mathbb{R}$ or $\mathbb{C}$.
\begin{defn}
A dg (left) Leibniz algebra $\g$ over $\k$ is a (left) Leibniz algebra in the category of cochain complexes of $\k$-vector spaces, i.e., a cochain complex $\vec{\g} = (\g^\bullet,d)$ together with a cochain map
$[-,-]_\g: \vec{\g} \otimes \vec{\g} \rightarrow \vec{\g}$,
called Leibniz bracket, satisfying the (left) Leibniz rule:
\[
 [x,[y,z]_\g]_\g = [[x,y]_\g,z]_\g + (-1)^{\abs{x}\abs{y}} [y,[x,z]_\g]_\g, \;\;\forall x,y,z \in \g^\bullet.
\]
\end{defn}

\begin{example}\label{Eg:Leibniz algebra}
Let $\h = (\h^\bullet,d,[-,-]_\h)$ be a dg Lie algebra equipped with a representation $\rho$ on a cochain complex $V=(V^\bullet,d)$. Then the
direct sum cochain complex $\h \oplus V = (\h^\bullet\oplus V^\bullet,d)$, equipped with the bilinear map $\{-,-\}$
defined by
\[
\{A+v, B+w\} \triangleq [A,B]_\h + \rho(A)w,\quad \forall A,B \in \h^\bullet, v,w \in V^\bullet,
\]
is a dg Leibniz algebra.
In particular, the semi-direct product $gl(V)\ltimes V$ for any vector space $V$, is a Leibniz algebra, which is called an omni-Lie algebra~\cite{Weinstein}, and is usually denoted by $ol(V)$. Furthermore, the section space of an omni-Lie algebroid~\cite{CL} is also a Leibniz algebra.
\end{example}

\begin{Rem}
  The notion of 2-term dg Leibniz algebras was investigated by Sheng and Liu in~\cite{SL0} as special Leibniz 2-algebras. It was shown in \emph{loc.cit.} that there is a one-to-one correspondence between 2-term dg Leibniz algebras and crossed modules of Leibniz algebras.
\end{Rem}

\begin{defn}
A left representation $V$ of a dg Leibniz algebra $\g$ (or a left $\g$-module) is a cochain complex $\vec{V} = (V^\bullet,d)$ equipped with a cochain map $l: \vec{\g} \rightarrow \End(\vec{V})$, called left action, such that
\[
l_{[x,y]_\g} = [l_{x},l_{y}] \triangleq l_x \circ l_y -(-1)^{\abs{x}\abs{y}}l_y \circ l_x,
\]
for all $x,y \in \g^\bullet$.

A representation of $\g$ (or a $\g$-module) is a left representation $V = (\vec{V},l)$, together with another cochain map $r: \vec{\g} \rightarrow \End(\vec{V})$, called the right action, such that the following conditions hold:
\begin{align*}
 r_{[x,y]_\g} &= [l_{x},r_{y}] = l_x \circ r_y - (-1)^{\abs{x}\abs{y}} r_y \circ l_x, & r_{x} \circ l_{y} &= -r_{x} \circ r_{y}.
\end{align*}
\end{defn}

Given a left $\g$-module $V = (\vec{V},l)$, there are two standard ways to extend $V$ to a $\g$-module: one is the symmetric $\g$-module with the right action being $r=-l$; the other is the antisymmetric $\g$-module with zero right action.

\begin{example}
  Let $\g = (\g^\bullet,d,[-,-]_\g)$ be a dg Leibniz algebra. Then
  \begin{align*}
    l_xy &= [x,y]_\g, & r_xy &= (-1)^{\abs{x}\abs{y}}[y,x]_\g,
  \end{align*}
  for all $x,y \in \g^\bullet$, gives rise to the adjoint representation of $\g$ on $(\g^\bullet,d)$.
\end{example}

\begin{example}
Let $\g$ be a dg Leibniz algebra. Both the Leibniz kernel
\[
 K = K(\g) \triangleq \operatorname{span}\{[x,y]_\g + (-1)^{\abs{x}\abs{y}}[y,x]_\g \mid x,y \in \g\}
\]
and the (left) center
\[
Z = Z(\g) \triangleq \{x \in \mathfrak{g} \mid [x,y]_\g = 0,\; \forall y \in \g\}
\]
are dg ideals of $\mathfrak{g}$. The Leibniz bracket induces a left representation $l$ of $\mathfrak{g}$ on $K$ (or $Z$):
\[
l_x(y) \triangleq [x,y]_\g \qquad\forall x \in \g, y \in K (\text{or}\;Z).
\]
Moreover, the Leibniz bracket reduces to a Lie bracket on both $\g/K$ and $\g/Z$.
\end{example}

In~\cite{AP}, Ammar and Poncin defined a cohomology theory for Loday infinity algebras, which generalizes the Loday-Pirashvili cohomology of Leibniz algebras~\cite{LP}. In particular, we have the following
\begin{defn}
Let $\mathfrak{g} = (\g^\bullet,d,[-,-]_\g)$ be a dg Leibniz algebra and $V = ((V^\bullet,d),l,r)$ a $\g$-module. The Loday-Pirashvili (bi)complex $C^\bullet(\g,V)$ of the $\g$-module $V$ consists of the following data:
\begin{itemize}
\item the underlying graded vector space
\[
 \oplus_nC^n(\g,V) = \oplus_{p+q=n}\Hom((\g^\bullet)^{\otimes p},V^\bullet)^q,
\]
where $\Hom((\g^\bullet)^{\otimes p},V^\bullet)^q \subset \Hom((\g^\bullet)^{\otimes p},V^\bullet)$ consists of elements of degree $q$;
\item the differential $D = \delta + d_{\LP}^V: C^n(\g,V) \to C^{n+1}(\g,V)$, where
\begin{compactenum}
 \item $\delta: \Hom((\g^\bullet)^{\otimes p},V^\bullet)^q \to \Hom((\g^\bullet)^{\otimes p},V^\bullet)^{q+1}$ is the internal differential that is specified by the following equation
     \[
       \delta(\eta)(x_1,\cdots,x_p) = (-1)^{\ast_{(i-1)}} \sum_{i=1}^p\eta(x_1,\cdots,dx_i,\cdots,x_p) - (-1)^{n}d\eta(x_1,\cdots,x_p),
     \]
     for all $\eta \in \Hom((\g^\bullet)^{\otimes p},V^\bullet)^q$ and $x_1,\cdots,x_p \in \g^\bullet$, where $\ast_0 = 0$ and $\ast_i = \abs{x_1}+\cdots+ \abs{x_i}$ for all $1\leq i\leq p$.
 \item $d^V_{\LP}: \Hom((\g^\bullet)^{\otimes p},V^\bullet)^q \to \Hom((\g^\bullet)^{\otimes (p+1)},V^\bullet)^q$, called the Loday-Pirashvili differential, is defined by
  \begin{align}\label{eq: Leibniz coboundary}
  &\quad(d^V_{\LP}\eta)(x_{1},\cdots,x_{p+1}) \notag \\
  &=\sum_{i=1}^{p}(-1)^{\abs{x_i}(\ast_{i-1}+n)+i-1}l_{x_{i}}\eta(x_{1},\cdots, \widehat{x_i}, \cdots,x_{p+1})+(-1)^{\abs{x_{p+1}}(n+\ast_p)+p+1} r_{x_{p+1}}\eta(x_{1}, \cdots, x_{p}) \notag \\
  &\quad + \sum_{1\leq i<j\leq p+1}(-1)^{i+\abs{x_i}(\abs{x_{i+1}}+\cdots+ \abs{x_{j-1}})} \eta(x_{1},\cdots, \widehat{x_i}, \cdots,\widehat{x_j},[x_{i},x_{j}]_\g, \cdots, x_{p+1}),
  \end{align}
for all $\eta \in \Hom((\g^\bullet)^{\otimes p},V^\bullet)^q$ and all $x_1,\cdots,x_{p+1} \in \g^\bullet$.
\end{compactenum}
\end{itemize}
The resulting cohomology, denoted by $HL^{\bullet}(\g,V)$, is called the Loday-Pirashvili cohomology of the $\mathfrak{g}$-module $V$.
\end{defn}

Next, we recall the 2-category of hemistrict Lie 2-algebras, which is, in fact, a 2-subcategory of 2-term weak $L_\infty$-algebras~\cite{Roy1}:
\begin{defn}\label{Def of weak Lie 2-algebra}
A hemistrict Lie 2-algebra $L$ is a 2-term dg Leibniz algebra $(\vec{L} = (L^{-1} \xrightarrow{d} L^0),[-,-])$, equipped with a symmetric bilinear map $h_{2}: L^0 \odot L^0 \rightarrow L^{-1}$ of degree $(-1)$, called alternator, such that
\begin{equation}\label{Eq: anticomm up homotopy}
[x,y]+(-1)^{\abs{x}\abs{y}}[y,x]=d(h_{2})(x,y),
\end{equation}
for all $x,y \in L^\bullet$, and
\begin{align}\label{Eq1: hemi2}
  [h_{2}(x,y),z] &= 0, \\
  [x,h_{2}(y,z)] &= h_{2}([x,y],z) + h_{2}(y,[x,z]), \label{Eq2: hemi2}
\end{align}
for all $x,y,z \in L^0$. Here and in the sequel, $\odot$ means the symmetric tensor product over the base field $\k$.

A morphism $f$ of hemistrict Lie 2-algebras from $L^\prime = (\vec{L}^\prime,[-,-]^\prime,h_2^\prime)$ to $L = (\vec{L},[-,-],h_2)$ consists of a cochain map $f_1: \vec{L}^\prime \rightarrow \vec{L}$ and a chain homotopy $f_2: L^{\prime 0} \otimes L^{\prime 0} \to L^{-1}$ controlling the compatibility between $f_1$ and the brackets, i.e.,
\begin{equation}\label{Eq: f1uptof2}
 [f_1(x^\prime),f_1(y^\prime)] - f_1([x^\prime,y^\prime]^\prime) = d(f_2)(x^\prime,y^\prime),
\end{equation}
for all $x^\prime,y^\prime \in L^{\prime\bullet}$, such that the following compatible conditions hold:
\begin{align}
  h_2(f_1(x^\prime),f_1(y^\prime)) - f_1(h^\prime_2(x^\prime,y^\prime)) &= f_2(x^\prime,y^\prime) + f_2(y^\prime,x^\prime), \label{Eq: h2andf2}\\
  [f_1(x^\prime),f_2(y^\prime,z^\prime)] - [f_1(y^\prime),f_2(x^\prime,z^\prime)] - [f_2(x^\prime,y^\prime),f_1(z^\prime)]
  &=  f_2([x^\prime,y^\prime]^\prime,z^\prime) + f_2(y^\prime,[x^\prime,z^\prime]^\prime) - f_2(x^\prime,[y^\prime,z^\prime]^\prime), \label{Eq: Jacobiandf2}
\end{align}
for all $x^\prime,y^\prime,z^\prime \in L^{\prime 0}$. Assume that $f^\prime$ is a morphism of hemistrict Lie 2-algebras from $L^{\prime\prime}$ to $L^\prime$. The composition $f \circ f^\prime: L^{\prime\prime} \to L$ of $f$ and $f^\prime$ is defined by
\begin{align*}
  (f\circ f^\prime)_1 &= f_1 \circ f_1^\prime, & (f\circ f^\prime)_2(x^{\prime\prime},y^{\prime\prime}) &= f_2(f_1^\prime(x^{\prime\prime}), f_1^\prime(y^{\prime\prime})) + f_1(f_2^\prime(x^{\prime\prime},y^{\prime\prime})),
\end{align*}
for all $x^{\prime\prime}, y^{\prime\prime} \in L^{\prime\prime}$.

Let $f$ and $g$ be two morphisms of hemistrict Lie 2-algebras from $L^\prime = (\vec{L}^\prime,[-,-]^\prime,h_2^\prime)$ to $L = (\vec{L},[-,-],h_2)$. A 2-morphism $\theta: f \Rightarrow g$ is a chain homotopy $\theta: L^{\prime 0} \to L^{-1}$ from $f_1$ to $g_1$, i.e.,
\[
 f_1(x^\prime) - g_1(x^\prime) = d(\theta)(x^\prime) = d\theta(x^\prime) + \theta(dx^\prime), \;\;\;\forall x^\prime \in L^{\prime\bullet},
\]
satisfying the following condition
\[
 f_2(x^\prime,y^\prime) - g_2(x^\prime,y^\prime) = [g_1(x^\prime),\theta(y^\prime)] + [\theta(x^\prime),f_1(y^\prime)] - \theta([x^\prime,y^\prime]^\prime) = [f_1(x^\prime),\theta(y^\prime)] + [\theta(x^\prime),g_1(y^\prime)] - \theta([x^\prime,y^\prime]^\prime),
\]
for all $x^\prime, y^\prime \in L^{\prime 0}$.
\end{defn}

\begin{Rem}
  For simplicity, our definition of hemistrict Lie 2-algebra is different from that defined by Roytenberg in~\cite{Roy1}, where the symmetry assumption on the alternator $h_2$ is weaker. However, one can apply symmetrization on the alternator of a hemistrict Lie 2-algebra in the sense of Roytenberg to obtain a hemistrict Lie 2-algebra with symmetric alternator as in the above definition.
\end{Rem}

Let $L = (\vec{L},[-,-],h_2)$ be a hemistrict Lie 2-algebra.
By forgetting the alternator $h_2$, it follows immediately that $(\vec{L},[-,-])$ is a dg Leibniz algebra. Consequently,
\begin{compactenum}
  \item $(L^0,[-,-])$ is a Leibniz algebra and $L^{-1}$ is an $L^0$-module.
  \item $(H^0(\vec{L}),[-,-])$ is also a Leibniz algebra and $H^{-1}(\vec{L})$ is a symmetric $H^0(\vec{L})$-module.
\end{compactenum}
According to~\cite{Roy1}, there associates a natural hemistrict Lie 2-algebra from a Leibniz algebra.
\begin{example}\label{Ex:Leibniz algebra}
Let $(\g,[-,-]_\g)$ be a Leibniz algebra with Leibniz kernel $K$. Let $h\colon \g \otimes \g \rightarrow K$ be the symmetric $K$-valued pairing defined by $h(x,y) = [x,y]_\g + [y,x]_\g$, whose composition with the degree shifting operator $[1]$ gives rise to a degree $(-1)$ bilinear map
\[
 h_2\colon \g \otimes \g \xrightarrow{h} K \xrightarrow{[1]} K[1].
\]
Then $L_\g := (K[1] \hookrightarrow \g, [-,-]_\g, h_2)$ is a hemistrict Lie 2-algebra.
\end{example}

\subsection{Representations of hemistrict Lie 2-algebras}
Now we are ready to introduce the notion of representations of hemistrict Lie 2-algebras:
\begin{defn}\label{Def of rep}
Let $L=(\vec{L},[-,-],h_{2})$ be a hemistrict Lie 2-algebra. A representation of $L$ (or an $L$-module structure) on a 2-term cochain complex $\vec{V} = (V^{-1} \xrightarrow{d} V^{0})$ is specified by the following data:
\begin{itemize}
\item two cochain maps, called left and right actions, respectively:
  \begin{align}\label{Equ. action}
   l\colon & L^{i}\otimes V^{j}\rightarrow V^{i+j}, & r\colon &V^{j}\otimes L^{i}\rightarrow V^{j+i},
  \end{align}
  where $i,j=-1,0$ and $V^{-2} = 0$.
\item a degree $(-1)$ linear map $h_v:~L^{0} \odot V^{0} \rightarrow V^{-1}$, called the action homotopy.
\end{itemize}
The triple $(l,r,h_v)$ is required to satisfy the following requirements:
\begin{compactenum}
  \item $l$ gives rise to a left representation of the dg Leibniz algebra $(\vec{L},[-,-])$ on $V^\bullet$, i.e.,
      \begin{equation}\label{Eq: lisleftrep}
       l_{x_1}l_{x_2}v - (-1)^{\abs{x_1}\abs{x_2}}l_{x_2}l_{x_1}v = l_{[x_1,x_2]}v,
      \end{equation}
    for all $x_1,x_2 \in L^\bullet$ and $v \in V^\bullet$, where $l_xv = l(x,v)$.
  \item The actions $l,r$ skew-commute up to homotopy, i.e.,
      \begin{align}\label{Eq: comm uptohomotopy}
       l_xv + (-1)^{\abs{x}\abs{v}}r_xv &= d(h_v)(x,v),
      \end{align}
  for all $x \in L^\bullet, v \in V^\bullet$, where $r_xv = (-1)^{\abs{x}\abs{v}}r(v,x)$.
  \item The actions $l,r$, the action homotopy $h_v$, and the structure maps $[-,-],h_2$ of $L$, satisfy the following compatible equations:
      \begin{align}\label{Eq: compatible equs}
        r_{x_2}h_v(x_1,v) &= 0, \\
        l_{h_2(x_1,x_2)}v &= 0, \\
        l_{x_1}h_v(x_2,v) &= h_v([x_1,x_2],v) + h_v(x_2,l_{x_1}v),\label{Eq: landhl}
      \end{align}
      for all $x_1,x_2 \in L^0$ and $v \in V^0$.
\end{compactenum}
\end{defn}

\begin{Rem}
  Since the bracket $[-,-]$ and the actions $l,r$ are all cochain maps, it follows that all Equations~\eqref{Eq: compatible equs}-\eqref{Eq: landhl} still hold if we replace $h_v$ (resp. $h_2$)  by $d(h_v)$ (resp. $d(h_2)$).
\end{Rem}

\begin{example}
  Let $L=(\vec{L},[-,-],h_{2})$ be a hemistrict Lie 2-algebra. There is a natural representation of $L$ on its underlying 2-term cochain complex $\vec{L}$ with $l = r = [-,-]$ and $h_v = h_2$. This representation, called the adjoint representation of $L$, will be denoted by $\Ad_L$.
\end{example}

\begin{example}\label{Ex: AdLg}
  Let $\g$ be a Leibniz algebra and $L_\g$ the hemistrict Lie 2-algebra as in Example~\ref{Ex:Leibniz algebra}.
  Each ordinary representation $(V^0,l,r)$ of the Leibniz algebra $\g$ determines a representation $V$ of $L_\g$ on the 2-term cochain complex $\vec{V}:= (V^\prime[1] \hookrightarrow V^0)$, where $V^{\prime}:=\{l_x v + r_x v \mid x \in \g, v \in V^0\}$, the left and right actions are induced by $l,r$, and the action homotopy is given by $h_v(x,v) = [1](l_x v) + [1](r_x v)$.
  In particular, the adjoint representation of $\g$ on $\g$ gives rise to the adjoint representation $\Ad_{L_\g}$ of $L_\g$ on $K[1] \to \g$.
\end{example}

\begin{example}\label{Ex: rep on ovs}
  Let $L=(\vec{L},[-,-],h_{2})$ be a hemistrict Lie 2-algebra. A representation of $L$ on an ordinary (or ungraded) vector space $V$ is, by definition, specified by two linear maps
  \begin{align*}
    l:~& L^0 \otimes V \rightarrow V, & r:~& V \otimes L^0 \rightarrow V,
  \end{align*}
  such that $(V,l,r)$ is a symmetric representation of the Leibniz algebra $(L^0,[-,-])$ and that $dL^{-1} \subset L^{0}$ acts trivially on $V$. It thus follows that $(V,l,r)$ is a symmetric representation of the Leibniz algebra $(H^0(\vec{L}),[-,-])$.
\end{example}

Analogous to situations in Leibniz algebras, we have some special representations of hemistrict Lie 2-algebras:
\begin{defn}
  Let $L$ be a hemistrict Lie 2-algebra. A representation $V = (\vec{V},l,r,h_v)$ of $L$ is said to be symmetric if $h_v = 0$, and antisymmetric if $r = 0$.
\end{defn}
It follows immediately from Equation~\eqref{Eq: comm uptohomotopy} that any symmetric representation of $L$ is of the form $(\vec{V},l,r=-l)$ and any antisymmetric representation of $L$ on a 2-term cochain complex $\vec{V}$ is specified by an ``exact" left representation in the following sense:
\[
  l_xv = d(h_v)(x,v), \;\quad \forall x \in L^\bullet, v \in V^\bullet.
\]
\begin{example}
  Any representation of a hemistrict Lie 2-algebra $L$ on an ordinary vector space is symmetric.
\end{example}

Recall that Lie algebra modules are one-to-one correspondent to abelian extensions of the given Lie algebra. For hemistrict Lie 2-algebras, we have the following
\begin{prop}\label{prop: equiv de module}
 Let $L=(\vec{L},[-,-],h_{2})$ be a hemistrict Lie 2-algebra and $\vec{V}:= (V^{-1} \xrightarrow{d} V)$ a 2-term complex. Then the linear maps $(l,r,h_v)$ as in Definition~\ref{Def of rep} define a representation of $L$ on $\vec{V}$ if and only if the semidirect product
\[
L\oplus V \triangleq (\overrightarrow{L \oplus V}, \{-,-\} \triangleq [-,-]+l+r, H \triangleq h_{2}+h_v)
\]
is a hemistrict Lie 2-algebra.
\end{prop}
To prove this proposition, we need the following result on the relationship between representations of a hemistrict Lie 2-algebra and representations of the associated dg Leibniz algebra.
\begin{lem}\label{Lemma V}
  Let $V = (\vec{V},l,r,h_v)$ be a representation of the hemistrict Lie 2-algebra $L = (\vec{L},[-,-],h_2)$. Then $(\vec{V},l,r)$ is a representation of the dg Leibniz algebra $(\vec{L},[-,-])$.
\end{lem}
\begin{proof}
 Since $l$ gives rise to a left representation of $(\vec{L},[-,-])$ on $\vec{V}$, it suffices to check that
 \begin{align*}
   r_{[x,y]} &= [l_{x},r_{y}], & r_x \circ l_y &= -r_x \circ r_y, \qquad \forall x,y \in L^\bullet,
 \end{align*}
or
\begin{align*}
  r(v,[x,y]) &= l(x,r(v,y)) - (-1)^{\abs{x}\abs{y}}r(l(x,v),y), &
  r_xl_yv &= -(-1)^{\abs{v}\abs{y}}r_xr_yv,\;\forall v \in V^\bullet.
\end{align*}
For the first one, we compute for any $v \in V^{-1} \oplus V^0$,
\begin{align*}
   r(v,[x,y]) &= (-1)^{\abs{v}(\abs{x}+\abs{y})}r_{[x,y]}v \qquad\qquad\qquad\qquad \text{by Equation~\eqref{Eq: comm uptohomotopy}} \\
   &= -l_{[x,y]}v + d(h_v)([x,y],v) \qquad\qquad\qquad \text{by Equation~\eqref{Eq: lisleftrep}}\\
   &= -l_xl_yv + (-1)^{\abs{x}\abs{y}} l_yl_xv + d(h_v)([x,y],v) \quad \text{by Equation~\eqref{Eq: comm uptohomotopy}} \\
   &= (-1)^{\abs{y}\abs{v}}l_xr_yv - l_xd(h_v)(y,v) - (-1)^{\abs{v}\abs{y}}r_yl_xv \\
   &\quad +(-1)^{\abs{x}\abs{y}}d(h_v)(y,l_xv) + d(h_v)([x,y],v) \quad \text{by Equation~\eqref{Eq: landhl}}\\
   &= (-1)^{\abs{y}\abs{v}}l_xr_yv - (-1)^{\abs{x}\abs{y}+\abs{y}(\abs{x}+\abs{v})} r_yl_xv \\
   &= l(x,r(v,y)) - (-1)^{\abs{x}\abs{y}}r(l(x,v),y)
\end{align*}
For the second one, by Equation~\eqref{Eq: comm uptohomotopy}, we have
\begin{align*}
  r_x(r_y(v)) &= -(-1)^{\abs{y}\abs{v}}r_x(l_yv) + r_x(d(h_v)(y,v)) \qquad \text{by Equation~\eqref{Eq: compatible equs}}\\
   &= -(-1)^{\abs{y}\abs{v}}r_x(l_y(v)).
\end{align*}
\end{proof}

As an immediate consequence, we have
\begin{Cor}\label{Cor V}
  Let $V = (\vec{V},l,r,h_v)$ be a representation of a hemistrict Lie 2-algebra $L = (\vec{L},[-,-],h_2)$. Then both $(V^{-1},l,r)$ and $(V^0,l,r)$ are representations of the Leibniz subalgebra $(L^0,[-,-])$ of the dg Leibniz algebra $(\vec{L},[-,-])$.
\end{Cor}

\begin{proof}[Proof of Proposition~\ref{prop: equiv de module}]
 Let $V = (\vec{V},l,r,h_v)$ be a representation of $L$. According to  Lemma~\ref{Lemma V}, $(\vec{V},l,r)$ is a representation of the dg Leibniz algebra $(\vec{L},[-,-])$. Thus the semidirect product
 \[
  (\overrightarrow{L \oplus V},\{-,-\}=[-,-]+l+r)
 \]
 is a dg Leibniz algebra. Meanwhile, it is clear that $\{-,-\}$ is skewsymmetric up to chain homotopy $H_2 = h_2 + h_v$. Hence, $L \oplus V = (\overrightarrow{L \oplus V},\{-,-\}=[-,-]+l+r, H_2 = h_2 + h_v)$ is a hemistrict Lie 2-algebra.

 Conversely, to reconstruct the semidirect product of $L$ by a 2-term cochain complex $\vec{V}$, one needs the two cochain maps $l,r$ in Equation~\eqref{Equ. action} to recover the underlying semidirect product of dg Leibniz algebra and the chain homotopy $h_v$ to control the skewsymmetry of the Leibniz bracket on $\vec{L} \oplus \vec{V}$. Moreover, it follows from straightforward verifications that the semidirect product $L \oplus V$ of hemistrict Lie 2-algebra can be completed by $(l,r,h_v)$ only if they give rise to a representation of $L$ on $\vec{V}$.
\end{proof}

Now we study morphisms of representations.
\begin{defn}
  Let $L = (\vec{L},[-,-],h_2)$ be a hemistrict Lie 2-algebra with $V = (\vec{V},l,r,h_v)$, $V^\prime = (\vec{V}^\prime, l^\prime,r^\prime,h_v^\prime)$ and $V^{\prime\prime} = (\vec{V}^{\prime\prime}, l^{\prime\prime}, r^{\prime\prime}, h_v^{\prime\prime})$ being its representations.
  \begin{compactenum}
  \item A morphism $\phi: V \to V^\prime$ is a triple $(\phi_1,\phi_l,\phi_r)$, where $\phi_1\colon \vec{V} \to \vec{V}^\prime$ is a cochain map, $\phi_l\colon L^0 \otimes V^0 \to V^{\prime -1}$ and $\phi_r\colon V^0 \otimes L^0 \to V^{\prime -1}$ are two chain homotopies, such that the following conditions hold:
  \begin{align}\label{Eq: flandl}
    l^\prime(x,\phi_1(v)) - \phi_1(l(x,v))  &= d(\phi_l)(x,v), \\
    r^\prime(\phi_1(v),x) - \phi_1(r(v,x))  &= d(\phi_r)(v,x), \notag \\
    h_v^\prime(x,\phi_1(v)) - \phi_1(h_v(x,v)) &= \phi_l(x,v) + (-1)^{\abs{x}\abs{v}}\phi_r(v,x), \notag
  \end{align}
  for all $x \in L^\bullet$ and $v \in V^\bullet$.
  \item  Given two morphisms of representations
   \begin{align*}
   \phi=(\phi_1,\phi_l,\phi_r): V &\rightarrow V^\prime,  & \psi = (\psi_1,\psi_l,\psi_r): V^\prime &\rightarrow V^{\prime\prime},
   \end{align*}
   their composition $\psi \circ \phi = ((\psi \circ \phi)_1,(\psi \circ \phi)_l,(\psi \circ \phi)_r): V \to V^{\prime\prime}$ is defined by
   \begin{align*}
     (\psi \circ \phi)_1 &= \psi_1 \circ \phi_1: \vec{V} \to \vec{V}^{\prime\prime}, \\
     (\psi \circ \phi)_l  &= \psi_l \circ (\id_{L^0} \otimes \phi_1) + \psi_1 \circ \phi_l: L^0 \otimes V^0 \to V^{\prime\prime -1}, \\
     (\psi \circ \phi)_r &= \psi_r \circ (\phi_1 \otimes \id_{L^0}) + \psi_1 \circ \phi_r: V^0 \otimes L^0 \to V^{\prime\prime -1}.
   \end{align*}
  \item Two representations $V,V^\prime$ are said to be isomorphic, if there exist morphisms $\phi\colon V \to V^\prime$ and $\psi\colon V^\prime \to V$ such that $\psi \circ \phi = \id_V := (\id_{\vec{V}},0,0)$ and $\phi \circ \psi = \id_{V^\prime} := (\id_{\vec{V}^\prime},0,0)$.
  \end{compactenum}
\end{defn}
It can be directly verified that the collection of representations of $L$ and their morphisms form a category $\Rep(L)$, called the category of representations of the hemistrict Lie 2-algebra $L$.
\begin{lem}\label{Key lemma on rep}
  Any object $V = (\vec{V},l,r,h_v)$ in $\Rep(L)$ is isomorphic to the associated symmetric representation $V_s = (\vec{V},l,-l,0)$.
\end{lem}
\begin{proof}
 It follows from straightforward verifications that the pair of morphisms
 \begin{align*}
   \phi &= (\id_{\vec{V}}, 0, -h_v)\colon  V \to V_s , & \psi&= (\id_{\vec{V}}, 0, h_v)\colon  V_s \to V
 \end{align*}
 gives rise to the desired isomorphisms.
\end{proof}
For this reason, we may view $V_s$ as a minimal model of $V$, which plays a central role in our construction of cohomology of $V$ in the subsequent section.

Finally, we consider pullback representations.
Let $f = (f_1,f_2)\colon L^\prime \to L$ be a morphism of hemistrict Lie 2-algebras.
By Equation~\eqref{Eq: f1uptof2}, $f_1$ does not preserve the brackets $[-,-]^\prime$ and $[-,-]$ strictly but up to a chain homotopy $f_2$. Thus, representations cannot be pulled back in general.
We make the following
\begin{defn}
  A representation $V = (\vec{V},l,r,h_v)$ of $L$ is said to be $f$-compatible (or compatible with $f$), if the left action $l$ vanishes along the image of the chain homotopy $f_2$, i.e.,
  \[
  l_{f_2(x^\prime,y^\prime)} = 0 \in \Hom(V^0,V^{-1}),
  \]
  for all $x^\prime,y^\prime \in L^{\prime 0}$.
\end{defn}

As an immediate consequence, we have
\begin{prop}
Let $f\colon L^\prime \to L$ be a morphism of hemistrict Lie 2-algebras and $V$ an $f$-compatible representation. Then there is a pullback representation $V^\prime$ of $L^\prime$ on $\vec{V}$ defined by
\begin{align*}
  l^\prime(x^\prime,v) &= l(f_1(x^\prime),v), & r^\prime(v,x^\prime) &= r(v,f_1(x^\prime)), &
  h_v^\prime(x^\prime,v) &= h_v(f_1(x^\prime),v),
\end{align*}
for all $x^\prime \in L^\prime$ and $v \in V^\bullet$. Furthermore, if $V$ is symmetric, so is $V^\prime$.
\end{prop}

\begin{example}
  Let $f = (f_1,f_2)\colon L^\prime \to L$ be a morphism of hemistrict Lie 2-algebras. Any representation of $L$ on an ordinary vector space $V$ can be pulled back by $f$ to obtain a pullback representation of $L^\prime$.
\end{example}

\section{Cohomology of hemistrict Lie 2-algebras}\label{sec: coh}
In this section, inspired by the functorial construction of standard cohomology of Courant-Dorfman algebras by Roytenberg~\cite{Roy2}, we associate a cochain complex $C^\bullet(L,V)$, also called standard complex, to each representation $V$ of a hemistrict Lie 2-algebra $L$. We prove that this assignment gives rise to a functor from the category $\Rep(L)$ of representations of $L$ to the category of cochain complexes on the one hand, and on the other hand a functor from the category of hemistrict Lie 2-algebras to the category of cochain complexes.

\subsection{The standard complex}
Let $L = (\vec{L},[-,-],h_2)$ be a hemistrict Lie 2-algebra.
The symmetric pairing $h_2\colon L^0 \otimes L^0 \rightarrow L^{-1}$ induces a graded skewsymmetric bilinear map $h_2\colon L^0[1] \wedge L^0[1] \rightarrow L^{-1}[1]$. It determines a graded Lie algebra structure on the vector space $L[1]$. Let
\[
 U(L[1]) = \oplus_{n\geq0}U(L[1])_{-n} = \oplus_{n\geq0}\oplus_{k=0}^{\lfloor\frac{n}{2}\rfloor} \left(\otimes^{(n-2k)}(L^0[1]) \otimes S^k(L^{-1}[1])\right)/R
\]
be the universal enveloping algebra of $L[1]$, where $R$ is the subspace generated by elements of the form
\begin{align*}
  x_1 &\otimes \cdots \otimes x_i \otimes x_{i+1} \otimes \cdots \otimes x_{n-2k} \otimes \alpha_1\odot\cdots\odot\alpha_k \\
     &+ x_1 \otimes \cdots \otimes x_{i+1} \otimes x_{i} \otimes \cdots \otimes x_{n-2k} \otimes \alpha_1\odot\cdots\odot \alpha_k \\
     &+ x_1 \otimes \cdots \otimes \widehat{x_{i}} \otimes \widehat{x_{i+1}} \otimes \cdots \otimes x_{n-2k} \otimes h_2(x_{i},x_{i+1})\odot\alpha_1\odot\cdots \odot\alpha_k,
\end{align*}
for all $x_1,\cdots,x_{n-2k} \in L^0[1], \alpha_1,\cdots,\alpha_k \in L^{-1}[1]$ and all $1 \leq i \leq n-2k-1, 0 \leq k \leq \lfloor\frac{p}{2}\rfloor$.

Let $V = (\vec{V},l,r,h_v)$ be a representation of $L$. Define
\begin{align*}
 C^\bullet(L,V) &= \oplus_{n \geq {-1}}C^n(L,V) = \oplus_{n\geq -1}(C^{n+1}(L,V^{-1}) \oplus C^n(L,V^0)) \\
    &= \oplus_{n\geq-1}(\Hom(U(L[1])_{-(n+1)},V^{-1}) \oplus \Hom(U(L[1])_{-n}, V^{0})),
\end{align*}
where our convention is that $C^{-1}(L,V^0) = 0$ and $U(L[1])_{-1} = 0$.
It follows that the degree $n$ component $C^n(L,V)$ consists of $(n+2)$-tuples
\[
 \omega = (\omega^{(p)}_0,\omega^{(p)}_1, \cdots, \omega^{(p)}_{\lfloor\frac{n-p}{2} \rfloor}),
\]
where $p = -1$ or $0$, and
\[
 \omega^{(p)}_k\colon \otimes^{(n-p-2k)}(L^0[1]) \otimes S^k(L^{-1}[1]) \rightarrow V^{p},
\]
satisfies the following weak symmetry properties:
\begin{align}\label{Eq: weak symmetry}
 &\quad \omega^{(p)}_k(\cdots,x_i,x_{i+1},\cdots \mid \cdots) + \omega^{(p)}_k(\cdots,x_{i+1},x_{i},\cdots \mid \cdots) \notag \\
 &= - \omega^{(p)}_{k+1}(\cdots,\widehat{x_i},\widehat{x_{i+1}},\cdots \mid h_2(x_i,x_{i+1}), \cdots),
\end{align}
for all $1 \leq i \leq n-p-2k-1$.

Next, we construct a differential on $C^\bullet(L,V)$. According to Lemma~\ref{Key lemma on rep}, the representation $V$ is isomorphic to its minimal model $V_s = (\vec{V},l,-l,0)$. For this reason, our construction of differential below only depends on this symmetric representation $V_s$. By Lemma~\ref{Lemma V}, $(\vec{V},l,-l)$ is a symmetric representation of the dg Leibniz algebra $(\vec{L}, [-,-])$. Moreover, we have the following
\begin{lem}\label{Lemma: differential}
The differential $D$ of the Loday-Pirashvili complex of the symmetric representation $(\vec{V},l,-l)$ of $(\vec{L},[-,-])$ determines a differential on $C^\bullet(L,V)$
\begin{equation}\label{Eq: Differential}
 D = \delta + d_{\LP}\colon C^n(L,V) \rightarrow C^{n+1}(L,V),
\end{equation}
defined for all $\omega = (\omega^{(p)}_0,\cdots, \omega^{(p)}_{\lfloor\frac{n-p}{2}\rfloor}) \in C^n(L;V)$,
\begin{align*}
  (\delta\omega)^{(-1)}_k(x_1,\cdots,x_{n+2-2k} &\mid \alpha_1,\cdots,\alpha_k) = \sum_{i=1}^k\omega^{(-1)}_{k-1}(d\alpha_i,x_1,\cdots,x_{n+2-2k}\mid \alpha_1,\cdots,\widehat{\alpha_i},\cdots,\alpha_k) \\
  (\delta\omega)^{(0)}_k(x_1,\cdots,x_{n+1-2k}&\mid \alpha_1,\cdots,\alpha_k) = \sum_{i=1}^k\omega^{(0)}_{k-1}(d\alpha_i,x_1,\cdots,x_{n+1-2k}\mid \alpha_1,\cdots,\widehat{\alpha_i},\cdots,\alpha_k) \\
  &\quad + (-1)^{n+1}d(\omega^{(-1)}_k(x_1,\cdots,x_{n+1-2k}\mid \alpha_1,\cdots, \alpha_k)), \\
  (d_{\LP}\omega)^{(0)}_k(x_1,\cdots,x_{n+1-2k} &\mid \alpha_1,\cdots,\alpha_k)
= (d_{\LP}^{V^0}\omega^{(0)}_k(\cdots \mid \alpha_1,\cdots,\alpha_k)) (x_1,\cdots,x_{n+1-2k}) \\
  &\quad + \sum_{i=1}^{n+1-2k}\sum_{j=1}^k(-1)^{i+1} \omega^{(0)}_k(\cdots,\widehat{x_i},\cdots \mid \cdots,[\alpha_j,x_i],\cdots), \\
(d_{\LP}\omega)^{(-1)}_k(x_1,\cdots,x_{n+2-2k} &\mid \alpha_1,\cdots,\alpha_k)
= d_{\LP}^{V^{-1}}(\omega^{(-1)}_k(\cdots \mid \alpha_1,\cdots,\alpha_k)) (x_1,\cdots,x_{n+2-2k}) \\
&\quad +(-1)^{n}\sum_{j=1}^{k} l_{\alpha_j}\omega_{k-1}^{(0)}(x_1,\cdots, x_{n+2-2k} \mid \alpha_1,\cdots, \widehat{\alpha_j}, \cdots,\alpha_k) \\
&\quad + \sum_{i=1}^{n+2-2k}\sum_{j=1}^k(-1)^{i+1} \omega^{(-1)}_k(\cdots,\widehat{x_i},\cdots \mid \cdots,[\alpha_j,x_i],\cdots),
\end{align*}
where we have viewed $\omega^{(p)}_k(\cdots \mid \alpha_1,\cdots,\alpha_k)$ as a degree $(n-p-2k)$ element of Loday-Pirashvili cochain complex of the symmetric $L^0$-module $(V^p,l,-l)$.
\end{lem}
\begin{proof}
 It suffices to show that the differential $D$ preserves the weak symmetry property~\eqref{Eq: weak symmetry}. In fact, for any $\omega \in C^n(L,V)$,
 \begin{align*}
    &\quad (\delta\omega)^{(-1)}_k(\cdots,x_s,x_{s+1},\cdots \mid \alpha_1,\cdots,\alpha_k) + (\delta\omega)^{(-1)}_k(\cdots,x_{s+1},x_s,\cdots \mid \alpha_1,\cdots,\alpha_k) \\
    &= \sum_{i=1}^k\omega^{(-1)}_{k-1}(d\alpha_i,\cdots,x_s,x_{s+1},\cdots\mid \cdots,\widehat{\alpha_i},\cdots) + \omega^{(-1)}_{k-1}(d\alpha_i,\cdots,x_{s+1},x_s,\cdots \mid \cdots,\widehat{\alpha_i},\cdots) \\
    &=-\sum_{i=1}^k\omega_{k}^{(-1)}(d\alpha_i,\cdots,\widehat{x_s}, \widehat{x_{s+1}}, \cdots\mid h_2(x_s,x_{s+1}),\cdots,\widehat{\alpha_i}, \cdots) \\
    &=-(\delta\omega)_{k+1}^{(-1)}(\cdots,\widehat{x_s}, \widehat{x_{s+1}}, \cdots\mid h_2(x_s,x_{s+1}),\cdots) + \omega_{k}^{(-1)}(dh_2(x_s,x_{s+1}),\cdots,\widehat{x_s}, \widehat{x_{s+1}},\cdots \mid \cdots ).
 \end{align*}
 Meanwhile, since $[h_2(x_s,x_{s+1}),x_i] = 0$ and $l(h_2(x_s,x_{s+1}),v) = 0$ for all $x_i \in L^0, v \in V^0$, it follows that
\begin{align*}
  &\quad (d_{\LP}\omega)^{(-1)}_k(\cdots,x_s,x_{s+1},\cdots \mid \cdots ) + (d_{\LP}\omega)^{(-1)}_k(\cdots,x_{s+1},x_s,\cdots \mid \cdots ) \\
  &= -(d_{\LP}\omega)^{(-1)}_{k+1}(\cdots,\widehat{x_s}, \widehat{x_{s+1}},\cdots \mid h_2(x_s,x_{s+1}),\cdots) + (-1)^s\omega_{k}^{(-1)}(\cdots,dh_2(x_s,x_{s+1}),\cdots\mid\cdots) \\
  &\quad + \sum_{i<s}(-1)^i(\omega_k^{(-1)}(\cdots,\widehat{x_i},\cdots, [x_i,x_{s}],x_{s+1},\cdots\mid\cdots)+\omega_k^{(-1)}(\cdots, \widehat{x_i},\cdots, x_s,[x_i,x_{s+1}],\cdots\mid\cdots)) \\
  &\quad + \sum_{i<s}(-1)^i(\omega_k^{(-1)}(\cdots,\widehat{x_i},\cdots, [x_i,x_{s+1}], x_{s},\cdots\mid\cdots)+\omega_k^{(-1)}(\cdots, \widehat{x_i},\cdots,x_{s+1},[x_i,x_{s}],\cdots\mid \cdots)) \\
  &= -(d_{\LP}\omega)^{(-1)}_{k+1}(\cdots,\widehat{x_s}, \widehat{x_{s+1}},\cdots \mid h_2(x_s,x_{s+1}),\cdots) + (-1)^s\omega_{k}^{(-1)}(\cdots,dh_2(x_s,x_{s+1}),\cdots\mid\cdots) \\
  &\quad -\sum_{i<s}(-1)^i\omega_k^{(-1)}(\cdots,\widehat{x_i},\cdots, \widehat{x_s},\widehat{x_{s+1}},\cdots\mid h_2([x_i,x_s],x_{s+1}),\cdots) \\
  &\quad -\sum_{i<s}(-1)^i\omega_k^{(-1)}(\cdots,\widehat{x_i},\cdots, \widehat{x_s},\widehat{x_{s+1}},\cdots\mid h_2(x_s,[x_i,x_{s+1}]),\cdots) \qquad \text{by Equation~\eqref{Eq2: hemi2}}\\
  &= -(d_{\LP}\omega)^{(-1)}_{k+1}(\cdots,\widehat{x_s}, \widehat{x_{s+1}},\cdots \mid h_2(x_s,x_{s+1}),\cdots) + (-1)^s\omega_{k}^{(-1)}(\cdots,dh_2(x_s,x_{s+1}),\cdots\mid\cdots) \\
  &\quad -\sum_{i<s}(-1)^i\omega_k^{(-1)}(\cdots,\widehat{x_i},\cdots, \widehat{x_s},\widehat{x_{s+1}},\cdots\mid [x_i,h_2(x_s,x_{s+1})],\cdots) \qquad\text{by Equations~\eqref{Eq: anticomm up homotopy},\eqref{Eq1: hemi2}}\\
  &= -(d_{\LP}\omega)^{(-1)}_{k+1}(\cdots,\widehat{x_s}, \widehat{x_{s+1}},\cdots \mid h_2(x_s,x_{s+1}),\cdots) + (-1)^s\omega_{k}^{(-1)}(\cdots,dh_2(x_s,x_{s+1}),\cdots\mid\cdots) \\
  &\quad -\sum_{i<s}(-1)^i\omega_k^{(-1)}(\cdots,\widehat{x_i},\cdots, \widehat{x_s},\widehat{x_{s+1}},\cdots\mid h_2(x_i,dh_2(x_s,x_{s+1})),\cdots) \qquad \text{by Equation~\eqref{Eq: weak symmetry}} \\
  &= -(d_{\LP}\omega)^{(-1)}_{k+1}(\cdots,\widehat{x_s}, \widehat{x_{s+1}},\cdots \mid h_2(x_s,x_{s+1}),\cdots) - \omega_{k}^{(-1)}(dh_2(x_s,x_{s+1}),\cdots,\widehat{x_s}, \widehat{x_{s+1}},\cdots \mid \cdots ).
\end{align*}
 Hence, we have proved that $(D\omega)^{(-1)}_k$ satisfies the weak symmetry property~\eqref{Eq: weak symmetry}, i.e.,
 \begin{align*}
  &\quad (D\omega)^{(-1)}_k(\cdots,x_s,x_{s+1},\cdots \mid \alpha_1,\cdots,\alpha_k) + (D\omega)^{(-1)}_k(\cdots,x_{s+1},x_s,\cdots \mid \alpha_1,\cdots,\alpha_k) \\
  &= -(D\omega)_{k+1}^{(-1)}(\cdots,\widehat{x_s}, \widehat{x_{s+1}}, \cdots\mid h_2(x_s,x_{s+1}),\cdots).
 \end{align*}
 By some similar computations (see~\cite{Cai}), one can easily show that $(D\omega)_k^{(0)}$ also satisfies the weak symmetry property~\eqref{Eq: weak symmetry}. This completes the proof.
\end{proof}

\begin{defn}\label{Def: coh}
  Let $L$ be a hemistrict Lie 2-algebra and $V = (\vec{V},l,r,h_v)$ a representation of $L$.
  We call $(C^{\bullet}(L,V),D)$ the standard complex of the hemistrict Lie $2$-algebra $L$ valued in $V$, whose cohomology $H^{\bullet}(L,V)$ is called the cohomology of the representation $V$ of $L$.
\end{defn}
\begin{Rem}
  When the alternator $h_2$ vanishes,  $L$ becomes a strict Lie $2$-algebra. In this case, the cohomology $H^\bullet(L,V)$ defined above is isomorphic to the generalized Chevalley-Eilenberg cohomology~\cites{BSZ, LSZ} of the representation $V^s = (\vec{V},l,-l)$ of the strict Lie $2$-algebra $L$.
\end{Rem}
\begin{example}\label{Ex: CD-algebra}
  Let $(\mathcal{R},\mathcal{E},\langle-,-\rangle,\partial,[-,-])$ be a Courant-Dorfman algebra, where $\mathcal{R}$ is a commutative algebra, $\mathcal{E}$ is an $\mathcal{R}$-module, $\langle-,-\rangle$ is an $\mathcal{R}$-valued symmetric $\mathcal{R}$-bilinear form, $\partial$ is an $\mathcal{E}$-valued derivation of $\mathcal{R}$, and $[-,-]$ is a Dorfman bracket on $\E$. All the data are subjected to several conditions (see~\cite{Roy2}). Let $d_{dR}: \mathcal{R} \to \Omega^1$ be the K\"{a}hler differential of the algebra $\mathcal{R}$. By the universality of $\Omega^1$, there is a $\mathcal{R}$-module morphism $\rho^{\ast}\colon \Omega^1 \to \E$ such that $\partial = \rho^\ast \circ d_{dR}$, which is called the coanchor map of $\mathcal{E}$.
  There exists a hemistrict Lie 2-algebra structure $(\{-,-\},h_2)$ on the $2$-term cochain complex $L(\mathcal{E}) = (\rho^{\ast}\colon \Omega^{1}[1] \to \mathcal{E})$ defined as follows:
   \begin{align*}
    \{e_{1},e_{2}\} & \triangleq  [e_{1}, e_{2}], & \{\alpha,e\} & \triangleq -\iota_{\rho(e)}d_{dR}\alpha,\\
    \{e,\alpha\} & \triangleq  \mathcal{L}_{\rho(e)}\alpha, &  h_{2}(e_{1},e_{2}) & \triangleq  d_{dR}\langle e_{1},e_{2}\rangle,
   \end{align*}
  for all $e_1,e_2 \in \E, \alpha \in \Omega^1[1]$, where $\rho: \mathcal{E} \to \Der(\mathcal{R})$ is the anchor map defined by $\rho(e)f = \langle e,\partial f \rangle$ for all $e \in \mathcal{E}$ and $f \in \mathcal{R}$.
  The anchor map $\rho$ determines a symmetric representation of the hemistrict Lie 2-algebra $L(\mathcal{E})$. It follows that the standard complex $C^{\bullet}(L(\mathcal{E}),\mathcal{R})$ of the hemistrict Lie $2$-algebra $L(\mathcal{E})$ as in Definition~\ref{Def: coh} coincides with the convolution dg algebra in~\cite{Roy2}.
\end{example}

\subsection{Cohomology of low orders}
Let $L = (\vec{L},[-,-],h_2)$ be a hemistrict Lie 2-algebra. Thus, $(\vec{L},[-,-])$ is a dg Leibniz algebra and $L^0$ is a Leibniz subalgebra. Let $V = (\vec{V},l,r,h_v)$ be a representation of $L$. We consider the cohomology $\oplus_{n\geq-1} H^n(L,V)$ of some lower orders.

For $n = -1$, $H^{-1}(L,V)$ equals the zeroth Loday-Pirashvili cohomology $HL^0(L^0,H^{-1}(\vec{V}))$ of the $L^0$-module $H^{-1}(\vec{V}) = \{u \in V^{-1} \mid du = 0\}$, i.e.,
\[
H^{-1}(L,V) = HL^0(L^0,H^{-1}(\vec{V})) = \{u \in V^{-1} \mid du = 0, l(x,u) = r(u,x) = 0, \;\forall x \in L^0\}.
\]

For $n = 0$, a zeroth-cocycle is a pair $(v,f)$, where $v \in V^0$ and $f \in \Hom(L^0,V^{-1})$, satisfying
\begin{align*}
  l(x,v) &= df(x), & l(\alpha,v) &= -f(d\alpha), & d_{\LP}^{V^{-1}}(f) &= 0.
\end{align*}
for all $x \in L^0,\alpha \in L^{-1}$. The two equations can be reinterpreted as
\[
 l(-,v) = -D(f)\colon L^\bullet \to V^\bullet.
\]
It follows that a zeroth-cocycle is a left $(\vec{L},[-,-])$-invariant element $v \in V^0$ up to a chain homotopy $f$. Moreover, $(v,f)$ is a coboundary if it is of the form $(du,l(-,u))$ for some $u \in V^{-1}$.

For $n=1$, a $1$-cocycle $\psi$ is a pair $(\psi_1,\psi_2)$, where
\begin{compactenum}
  \item $\psi_1: \vec{L} \to \vec{V}$ is a cochain map, i.e., the following diagram commutes:
      \[
      \xymatrix{
         L^{-1} \ar[d]_-{d} \ar[r]^-{\psi_1} & V^{-1} \ar[d]^-{d} \\
         L^0 \ar[r]^-{\psi_1} & V^0;
      }
      \]
  \item $\psi_2 \in HL^2(L^0,V^{-1})$, satisfying the weak symmetry~\eqref{Eq: weak symmetry}, is the chain homotopy such that $\psi_1: \vec{L} \rightarrow \vec{V}$ is a derivation of the dg Leibniz algebra $(\vec{L},[-,-])$ up to homotopy, i.e.,
      \begin{align*}
        &l(x,\psi_1(y)) + r(\psi_1(x),y) - \psi_1([x,y]) = dh_v(\psi_1(x),y) - d\psi_2(x,y), \\
        &l(\alpha,\psi_1(x)) + r(\psi_1(\alpha),x) - \psi_1([\alpha,x]) = \psi_2(d\alpha,x) - h_v(d\psi_1(\alpha),x),
      \end{align*}
      for all $x,y \in L^0$ and $\alpha \in L^{-1}$.
\end{compactenum}
It thus follows that a $1$-cocycle $\psi$ is a dg derivation $\psi_1$ of the dg Leibniz algebra $(\vec{L},[-,-])$ up to homotopy valued in $\vec{V}$. A 1-coboundary, called an inner dg derivation, is of the form
\[
(\psi_1,\psi_2) = (l(-,v)-df,d_{\LP}^{V^{-1}}(f)),
\]
for some $v \in V^0$ and $f \in \Hom(L^0,V^{-1})$.

In particular, let $\g$ be a Leibniz algebra and $L_\g = (K[1]\hookrightarrow \g,[-,-]_\g,h_2)$ the hemistrict Lie 2-algebra as in Example~\ref{Ex:Leibniz algebra}. We have
\begin{prop}
 There exists a natural injection sending $HL^1(\g,\g)$ into $H^1(L_\g,\Ad_{L_\g})$.
\end{prop}
\begin{proof}
According to Loday and Pirashvili~\cite{LP}, $HL^1(\g,\g) = \Der(\g,\g)/\{\text{inner derivations}\}$. It suffices to assign a(n) (inner) derivation of $(K[1]\oplus\g, [-,-]_\g)$ up to homotopy valued in $K[1]\oplus \g$ to each (inner) derivation of $\g$ valued in $\g$.

On the one hand, each $\g$-valued derivation of $\g$ is, by definition, a linear map $\phi_1: \g \to \g$ satisfying
\[
 \phi_1([x,y]_\g) = [\phi_1(x),y]_\g + [x,\phi_1(y)]_\g,\qquad \forall x,y,\in \g.
\]
It follows that $\phi_1$ maps the Leibniz kernel $K$ to itself. Thus, it extends to a cochain map $\phi_1: K[1] \hookrightarrow \g \to K[1] \hookrightarrow \g$. Define $\phi_2: \g \otimes \g \rightarrow K[1]$ by
\[
 \phi_2(x,y) = h_2(\phi_1(x),y) = [\phi_1(x),y]_\g + [y,\phi_1(x)]_\g,
\]
for all $x,y \in \g$. It follows from a straightforward verification that $\phi = (\phi_1,\phi_2)$ is a 1-cocycle of the adjoint representation $\Ad_{L_\g}$ of $L_\g$.

On the other hand, each inner derivation $l_x$ for some $x \in \g$ gives rise to an inner dg derivation $(r_x,0)$ of $\Ad_{L_\g}$.
\end{proof}

For $n=2$, a 2-cocycle $\omega$ is a quadruple
\begin{align}\label{2-cocycle}
 \omega_1^{(0)}:& L^{-1} \to V^0, & \omega_{0}^{(0)}:& L^0 \otimes L^0 \to V^0, & \omega_{1}^{(-1)}:& L^{0} \wedge L^{-1} \to V^{-1}, & \omega_{0}^{(-1)}: & (L^{0})^{\otimes 3} \to V^{-1},
\end{align}
satisfying the following conditions:
\begin{compactenum}
  \item $\omega_{0}^{(0)}$ and $\omega_{0}^{(-1)} \in HL^3(L^0,V^{-1})$ satisfy the weak symmetry conditions~\eqref{Eq: weak symmetry}.
  \item $\omega_1^{(0)}$ is $(\vec{L},[-,-])$-invariant up to homotopy, i.e.,
     \begin{align}\label{Eq1:2-cocycle}
     \omega_1^{(0)}([\alpha,x]) - r(\omega_1^{(0)}(\alpha),x) &= d\omega_{1}^{(-1)}(\alpha,x) - \omega_{0}^{(0)}(d\alpha,x) - dh_v(\omega_1^{(0)}(\alpha),x),  \\  \label{Eq2: 2-cocycle}
     l(\alpha,\omega_1^{(0)}(\beta)) - r(\omega_1^{(0)}(\alpha),\beta) &= \omega_{1}^{(-1)}(\alpha,d\beta) - \omega_{1}^{(-1)}(d\alpha,\beta) - h_v(\omega_1^{(0)}(\alpha),d\beta),
     \end{align}
     for all $x \in L^0$ and $\alpha,\beta \in L^{-1}$.
  \item $\omega_{0}^{(0)}$ and $\omega_{1}^{(-1)}$ are 2-cocycles up to homotopy of the representation of the graded Leibniz algebra $(L^\bullet,[-,-])$ on $V^\bullet$, i.e.,
      \begin{align}\label{Eq3:2-cocycle}
       -d\omega_{0}^{(-1)}(x,y,z) &= d_{\LP}^{V^0}(\omega_{0}^{(0)})(x,y,z)  \\
       h_v(d\omega_1^{(-1)}(x\mid\alpha),y) - \omega_{0}^{(-1)}(d\alpha,x,y) &= l(x,\omega_{1}^{(-1)}(y \mid \alpha)) + r(\omega_{1}^{(-1)}(x \mid \alpha),y) + l(\alpha,\omega_{0}^{(0)}(x,y)) \notag \\  \label{Eq4:2-cocycle}
       &\quad - \omega_{1}^{(-1)}([x,y] \mid \alpha) + \omega_{1}^{(-1)}(y\mid [\alpha,x]) - \omega_{1}^{(-1)}(x\mid [\alpha,y]),
      \end{align}
      for all $x,y \in L^0$ and $\alpha \in L^{-1}$.
\end{compactenum}

Given a hemistrict Lie 2-algebra $L$ and an $L$-module $V$, recall that an abelian extension of $L$ by $V$ in the category of hemistrict Lie 2-algebras (resp. in the category of weak Lie 2-algebras) is a short exact sequence of hemistrict Lie 2-algebras (resp. weak Lie 2-algebras)
\[
 0 \rightarrow V \rightarrow E \rightarrow L \rightarrow 0,
\]
such that the sequence splits as graded vector spaces, the brackets on $V$ is trivial and the action of $L$ on $V$ is the prescribed one. Two such extensions $E$ and $E^\prime$ are isomorphic if there exists a morphism of hemistrict Lie 2-algebras (resp. weak Lie 2-algebras) from $E$ to $E^\prime$ which is compatible with the identity on $V$ and on $L$.
Analogous to abelian extensions of Leibniz algebras~\cites{CFV,LP}, we have the following
\begin{lem}\label{Lemma: extension}
   Each second cohomology class $[\omega] \in H^2(L,V)$ of a representation $V$ of the hemistrict Lie 2-algebra $L$ gives rise to an equivalence class of abelian extensions of $L$ by $V$ in the category of weak Lie 2-algebras.
\end{lem}
\begin{proof}
 Let $\omega = (\omega_0^{(0)},\omega_1^{(0)},\omega_1^{(-1)},\omega_0^{(-1)})$ be a 2-cocycle as in Equation~\eqref{2-cocycle} of the representation $V$. Consider the following binary operation $\{-,-\}$ defined by
 \begin{align*}
  \{(x,v),(y,w)\} &:= ([x,y], l(x,w) - l(y,v) + \omega_0^{(0)}(x,y)), \\
  \{(x,v),(\beta,w)\} &:= ([x,\beta], l(x,w) - l(\beta,v) + \omega_1^{(-1)}(x,\beta)), \\
  \{(\alpha,v), (y,w)\} &:= ([\alpha,y], l(\alpha,w) - l(y,v) + \omega_1^{(-1)}(\alpha,y)), \\
  \{(\alpha,v),(\beta,w)\} &:= (0,l(\alpha,w) - l(\beta,v)),
 \end{align*}
for all $x,y \in L^0, \alpha,\beta\in L^{-1}$ and $v,w \in V^{-1}\oplus V^0$,  on the cochain complex $\overrightarrow{L\oplus V} = (L^\bullet \oplus V^\bullet, d_\omega)$, where the differential $d_\omega$ is specified by
 \begin{align*}
   d_\omega(\alpha) &= d\alpha - \omega_1^{(0)}(\alpha), & d_\omega(u) = du,
 \end{align*}
for all $\alpha \in L^{-1}$ and $u \in V^{-1}$. Using Equations~\eqref{Eq1:2-cocycle} and~\eqref{Eq2: 2-cocycle}, one can easily verify that $\{-,-\}$ is indeed a $d_\omega$-cochain map.

Define a degree $(-1)$ map $H_3\colon (L^\bullet\oplus V^\bullet)^{\otimes 3} \to L^\bullet \oplus V^\bullet$ by
\[
 H_3((x,u),(y,v),(z,w)) := (0,\omega_0^{(-1)}(x,y,z)).
\]
By Equations~\eqref{Eq3:2-cocycle} and~\eqref{Eq4:2-cocycle}, it can also be checked directly that $\{-,-\}$ satisfies the Jacobi identity up to the chain homotopy $H_3$.

Finally, define a degree $(-1)$ bilinear map $H_2$ on $L^\bullet \oplus V^\bullet$ by
\[
 H_2((x,u),(y,w)) := (h_2(x,y),0),\qquad \forall x,y \in L^0, u,w \in V^\bullet.
\]
It is clear that
\begin{align*}
 \{(x,u),(y,w)\} + \{(y,w),(x,u)\} &= d_\omega(H_2((x,u),(y,w))), \\
 \{(x,u),(\beta,w)\} + \{(\beta,w),(x,u)\} &= H_2((x,u),d_\omega(\beta,w)),\\
 \{(\alpha,u),(y,w)\} + \{(y,w),(\alpha,u)\} &= H_2(d_\omega(\alpha,u),(y,w)).
\end{align*}
Hence, $L \oplus V = (L^\bullet \oplus V^\bullet,d_\omega,\{-,-\},H_2,H_3)$ is a weak Lie 2-algebra. We obtain an extension of $L$ by $V$
\begin{equation}\label{SES of hLie2}
 0 \to V \xrightarrow{i = (i_1,i_2)} L \oplus V \xrightarrow{\pr = (\pr_1,\pr_2)} L \to 0,
\end{equation}
where $i_1(v) = (0,v),\pr_1(x,v) = x$ for all $v \in V^\bullet$ and $x \in L^\bullet$, and both chain homotopy $i_2$ and $\pr_2$ vanish.

Furthermore, any splitting $f=(f_1,f_2)$ of short exact sequence~\eqref{SES of hLie2} of weak Lie 2-algebras, if exists, is of the form
\begin{align*}
  f_1 &= \id \oplus \psi_1: L \rightarrow L \oplus V, & f_2 &= (0,\psi_2): L \otimes L \rightarrow L \oplus V,
\end{align*}
where $\psi = (\psi_1,\psi_2) \in C^1(L,V)$ is a 1-cochain. Moreover, it follows from straightforward computations that $f$ is indeed a splitting, i.e., $f$ is a morphism of weak Lie 2-algebras, if and only if $D(\psi) = \omega$.
This completes the proof.
\end{proof}

\begin{Rem}
  When the alternator $h_2$ of $L$ vanishes, $L$ becomes a strict Lie 2-algebra and the alternator on $L \oplus V$ also vanishes. As a consequence, we rediscover the cohomological description of abelian extensions of strict Lie 2-algebras in~\cite{LSZ}.
\end{Rem}
Consider the subset $\widetilde{H}^2(L,V) \subset H^2(L,V)$ consisting of cohomology classes $[\omega]$, which have a representative $\omega = (\omega_1^{(0)},\omega_0^{(0)},\omega_1^{(-1)},\omega_0^{(-1)})$ such that $\omega_0^{(-1)} = 0$. By the argument in the proof of Lemma~\ref{Lemma: extension}, each element in $\widetilde{H}^2(L,V)$ gives rise to an extension class of $L$ by $V$ in the category of hemistrict Lie 2-algebras. Conversely, it is easy to see that each extension arises in this way.
In summary, we have
\begin{prop}\label{prop: h2}
  The subset $\widetilde{H}^2(L,V)$ is isomorphic to the set of abelian extension classes of $L$ by $V$ in the category of hemistrict Lie 2-algebras.
\end{prop}
\begin{Rem}
  It is natural to consider non-abelian extensions of hemistrict Lie 2-algebras. In~\cite{SZ}, Sheng and Zhu interpreted non-abelian extensions of Lie algebras as morphisms of Lie 2-algebras. Recently, Liu, Sheng and Wang~\cite{LSW} studied non-abelian extensions of Leibniz algebras by morphisms of Leibniz 2-algebras. Thus, it is expected that non-abelian extensions of hemistrict Lie 2-algebras would be described by morphisms of hemistrict Lie 3-algebras. We will investigate this problem somewhere else.
\end{Rem}

\subsection{Functoriality}
In this section, we prove that the construction of standard complexes of hemistrict Lie 2-algebras is functorial.
First of all, we fix a hemistrict Lie 2-algebra $L = (\vec{L},[-,-],h_2)$ and prove that the assignment of standard complexes to representations of $L$ is functorial.
More precisely, we prove the following
\begin{prop}\label{prop: functoriality1}
Assume that $\phi = (\phi_1,\phi_l,\phi_r)\colon V \to V^\prime$ is a morphism of representations of $L$. Then there associates a cochain map
\[
   \phi_\ast\colon \oplus_nC^n(L,V) \longrightarrow \oplus_nC^n(L,V^\prime),
\]
defined by
\begin{align}\label{Eq: phiast0}
  (\phi_\ast\omega)_k^{(0)}(x_1,\cdots,x_{n-2k}&\mid\alpha_1,\cdots,\alpha_k) = \phi_1(\omega_k^{(0)}(x_1,\cdots,x_{n-2k}\mid\alpha_1,\cdots,\alpha_k)), \\
  (\phi_\ast\omega)_k^{(-1)}(x_1,\cdots,x_{n-2k+1}&\mid\alpha_1,\cdots,\alpha_k) = \phi_1(\omega_k^{(-1)}(x_1,\cdots,x_{n-2k+1}\mid\alpha_1,\cdots,\alpha_k)) \notag \\
   &\quad -\sum_{i=1}^{n-2k+1}(-1)^{n+i}\phi_l(x_i,\omega_k^{(0)}(x_1,\cdots, \widehat{x_i},\cdots,x_{n-2k+1}\mid\alpha_1,\cdots,\alpha_k)), \label{Eq: phiast-1}
\end{align}
for all $\omega \in C^n(L,V)$, and all $x_1,\cdots,x_{n-2k+1} \in L^0, \alpha_1,\cdots,\alpha_k \in L^{-1}$.
\end{prop}
\begin{proof}
 It follows from straightforward verifications that both $(\phi_\ast\omega)_k^{(0)}$ and $(\phi_\ast\omega)_k^{(-1)}$ satisfy the weak symmetry property~\eqref{Eq: weak symmetry}. Thus $\phi_\ast$ is well-defined.

 Now we show that $\phi_\ast$ is a cochain map. For each $\omega \in C^n(L,V)$, we compute
 \begin{align*}
  &\quad (\phi_\ast D\omega)_k^{(0)}(x_1,\cdots,x_{n-2k+1}\mid\alpha_1,\cdots, \alpha_k) =  \phi_1((D\omega)_k^{(0)}(x_1,\cdots,x_{n-2k+1}\mid\alpha_1,\cdots,\alpha_k)) \\
  &= \sum_{j=1}^k\phi_1(\omega_{k-1}^{(0)}(d\alpha_j,x_1,\cdots,x_{n-2k+1}\mid \cdots, \widehat{\alpha_j},\cdots)) + (-1)^{n+1}\phi_1(d\omega_k^{(-1)}(x_1,\cdots,x_{n-2k+1}\mid \alpha_1,\cdots,\alpha_k)) \\
  &\quad+ \sum_{i=1}^{n-2k+1} (-1)^{i-1}\phi_1(l(x_i,\omega^{(0)}_k(\cdots, \widehat{x_i},\cdots \mid \cdots))) + \sum_{i<j}(-1)^i \phi_1(\omega_{k}^{(0)}( \widehat{x_i},\cdots,\widehat{x_j},[x_i,x_j] \mid \cdots))\\
  &\quad + \sum_{i=1}^{n-2k+1}\sum_{j=1}^k(-1)^{i+1} \phi_1(\omega^{(0)}_k(\cdots,\widehat{x_i},\cdots \mid \cdots,[\alpha_j,x_i],\cdots)), \qquad \text{by Equations~\eqref{Eq: phiast0},\eqref{Eq: phiast-1}}\\
  &= \sum_{j=1}^k (\phi_\ast\omega)_{k-1}^{(0)}(d\alpha_j,x_1,\cdots,x_{n-2k+1} \mid \cdots, \widehat{\alpha_j},\cdots) + (-1)^{n+1}d(\phi_\ast\omega)_k^{(-1)}(x_1,\cdots,x_{n-2k+1}\mid \alpha_1,\cdots,\alpha_k) \\
  &\quad+ \sum_{i=1}^{n-2k+1} (-1)^{i-1} (\phi_1(l(x_i,\omega^{(0)}_k(\cdots, \widehat{x_i},\cdots \mid \cdots))) + d\phi_l(x_i,\omega^{(0)}_k(\cdots, \widehat{x_i},\cdots \mid \cdots)))\\
  &\quad + \sum_{i<j}(-1)^i (\phi_\ast\omega)_{k}^{(0)}( \widehat{x_i},\cdots,\widehat{x_j},[x_i,x_j] \mid \cdots) \\
  &\quad + \sum_{i=1}^{n-2k+1}\sum_{j=1}^k(-1)^{i+1} (\phi_\ast\omega)^{(0)}_k(\cdots,\widehat{x_i},\cdots \mid \cdots,[\alpha_j,x_i],\cdots) \qquad  \text{by Equation~\eqref{Eq: flandl}}\\
  &= (D\phi_\ast\omega)_k^{(0)}(x_1,\cdots,x_{n-2k+1} \mid \alpha_1,\cdots, \alpha_k),
 \end{align*}
for all $x_1,\cdots,x_{n-2k+1} \in L^0, \alpha_1,\cdots,\alpha_k \in L^{-1}$. Similarly, one can compute directly that
\[
 (\phi_\ast D\omega)_k^{(-1)}(x_1,\cdots,x_{n-2k+2}\mid \alpha_1,\cdots,\alpha_k) = (D\phi_\ast\omega)_k^{(-1)}(x_1,\cdots,x_{n-2k+2}\mid \alpha_1,\cdots,\alpha_k).
\]
\end{proof}

Next, we prove that the construction of standard complexes is functorial with respect to morphisms of hemistrict Lie 2-algebras.
More precisely, we have the following
\begin{prop}\label{prop: pullback}
  Let $f=(f_1,f_2): L^\prime \to L$ be a morphism of hemistrict Lie 2-algebras and $V = (\vec{V},l,r)$ an $f$-compatible representation of $L$. Denote by $V^\prime = (\vec{V},l^\prime,r^\prime)$ the pullback representation of $L^\prime$ on $\vec{V}$. Then $f$ induces a morphism of standard complexes
  \[
   f^\ast: \oplus_nC^n(L,V) \rightarrow \oplus_n C^n(L^\prime,V^\prime),
  \]
  defined by for each $\omega \in C^n(L,V)$,
  \begin{align*}
  &\quad (f^\ast\omega)_k^{(p)}(x^\prime_1,\cdots,x^\prime_{n-2k-p} \mid \alpha^\prime_1,\cdots, \alpha^\prime_j) \\
  &= \omega_k^{(p)}(f_1(x^\prime_1),\cdots,f_1(x^\prime_{n-2k-p})\mid f_1(\alpha^\prime_1),\cdots,f_1(\alpha_k^\prime)) \\
  &\quad - \sum_{q=1}^{\lfloor\frac{n-p}{2}\rfloor - k} \sum_{\substack{i_1<j_1;\cdots;i_q<j_q \\ i_1<\cdots<i_q}} \sgn(i_1<j_1;\cdots;i_q<j_q) \omega^{(p)}_{k+q}(\widehat{f_1(x^\prime_{i_1})},\cdots, \widehat{f_1(x^\prime_{j_1})},\cdots,\widehat{f_1(x^\prime_{i_q})},\cdots, \widehat{f_1(x^\prime_{j_q})} \\
  &\qquad\qquad\qquad\qquad\qquad\qquad\qquad\qquad\qquad\qquad \mid f_2(x^\prime_{i_1},x^\prime_{j_1}),\cdots, f_2(x^\prime_{i_q}, x^\prime_{j_q}), f_1(\alpha^\prime_1), \cdots, f_1(\alpha_k^\prime)),
\end{align*}
for all $x^\prime_1,\cdots,x^\prime_{n-2k-p} \in L^{\prime 0}, \alpha^\prime_1, \cdots,\alpha^\prime_k \in L^{\prime -1}$, where
\begin{equation}\label{Eq: sgn}
\sgn(i_1<j_1;\cdots;i_q<j_q) \triangleq (-1)^{\sum_{a=1}^q(i_{a}+j_{a}) + \sum_{a<b}\operatorname{int}((i_{a},j_{a}),(i_{b},j_{b}))},
\end{equation}
and $\operatorname{int}((i_{a},j_{a}),(i_{b},j_{b}))$ is the mod 2 intersection number of the two pairs $(i_{a},j_{a})$ and $(i_{b},j_{b})$, which either equals $1$ if $i_{a}<i_{b}<j_{a}<j_{b}$, or vanishes otherwise.
\end{prop}
As a consequence, when $V$ is the trivial representation of $L$ on the base field $\k$, we have
\begin{thm}\label{Thm: Functoriality}
  The assignment $L \to C^\bullet(L,\k), f \mapsto f^\ast$ is a contravariant functor from the category of hemistrict Lie 2-algebras to the category of cochain complexes.
\end{thm}

To prove Proposition~\ref{prop: pullback}, one needs, on the one hand, to verify that $f^\ast$ is well defined, i.e., for any $\omega \in C^n(L,V)$, $f^\ast\omega$ satisfies the following weak symmetry property:
\begin{align}\label{Eq: weak symm for fast}
  &\quad (f^\ast\omega)_k^{(p)}(\cdots,x^\prime_i,x^\prime_{i+1},\cdots \mid \alpha^\prime_1,\cdots, \alpha^\prime_j) + (f^\ast\omega)_k^{(p)}(\cdots,x^\prime_{i+1},x^\prime_i,\cdots \mid \alpha^\prime_1,\cdots, \alpha^\prime_j)  \notag \\
  &= -(f^\ast\omega)_{k+1}^{(p)}(\cdots,\widehat{x^\prime_i}, \widehat{x^\prime_{i+1}},\cdots \mid h^\prime_2(x^\prime_{i},x^\prime_{i+1}),  \alpha^\prime_1,\cdots, \alpha^\prime_j),
\end{align}
and on the other hand, to prove that $f^\ast$ is a cochain map, i.e.,
\begin{equation}\label{Eq: fastcochain}
(Df^\ast\omega)_k^{(p)}(x_1,\cdots,x_{n-2k-p+1}\mid\alpha_1,\cdots,\alpha_k) = (f^\ast D\omega)_k^{(p)}(x_1,\cdots,x_{n-2k-p+1}\mid\alpha_1,\cdots,\alpha_k).
\end{equation}
In fact, both Equation~\eqref{Eq: weak symm for fast} and Equation~\eqref{Eq: fastcochain} follow from a straightforward but tedious calculation by using the following equations
\begin{align}\label{Eq: f2}
    h_2(f_1(x_1^\prime),f_1(x_{2}^\prime)) - f_1(h_2^\prime(x_1^\prime, x^\prime_{2})) &= f_2(x_1^\prime, x_{2}^\prime) + f_2(x_{2}^\prime, x_{1}^\prime), \\
    [f_1(x_1^\prime),f_1(x_2^\prime)] - f_1([x_1^\prime,x_2^\prime]^\prime) &= d(f_2)(x_1^\prime,x_2^\prime); \label{Eq: f2ischainhomotopy}\\
  [f_1(x_1^\prime),f_2(x_2^\prime,x_3^\prime)] - [f_1(x_2^\prime), f_2(x_1^\prime,x_3^\prime)] &- [f_2(x_1^\prime,x_2^\prime),f_1(x_3^\prime)] \notag \\
  &= f_2([x_1^\prime,x_2^\prime]^\prime,x_3^\prime) + f_2(x_2^\prime, [x_1^\prime,x_3^\prime]^\prime) - f_2(x_1^\prime,[x_2^\prime,x_3^\prime]^\prime), \label{Eq: f2andbrackets}
\end{align}
for all $x_1^\prime, x_{2}^\prime, x_3^\prime \in L^{\prime \bullet}$, by the definition of morphisms of hemistrict Lie 2-algebras, and
\begin{align}\label{Eq: pullback rep}
  l^\prime_{x^\prime}v &= l_{f_1(x^\prime)}v, \quad \forall x^\prime \in L^{\prime\bullet}, v \in V^\bullet,
\end{align}
by the definition of pullback representations. To save space and time, we omit the proof.
However, in order to see how the above equations are involved in calculations, we verify Equation~\eqref{Eq: weak symm for fast} in the case that $\psi = (\psi_1,\psi_2) \in C^1(L,V)$ and Equation~\eqref{Eq: fastcochain} in the case that $\omega = (\omega_0^{(0)},\omega_1^{(0)},\omega_1^{(-1)},\omega_0^{(-1)}) \in C^2(L,V)$, respectively, by proving the following
\begin{align*}
  (f^\ast\psi)_2(x_1^\prime,x_2^\prime) +  (f^\ast\psi)_2(x_2^\prime,x_1^\prime) &= -(f^\ast\psi)_1(h_2^\prime(x_1^\prime,x_2^\prime)); \\
  (f^\ast D\omega)_0^{(0)}(x^\prime_1,x^\prime_2,x^\prime_3) &= (D f^\ast\omega)_0^{(0)} (x^\prime_1,x^\prime_2,x^\prime_3).
\end{align*}
For the first one, we compute
\begin{align*}
    &\quad (f^\ast\psi)_2(x_1^\prime,x_2^\prime) +  (f^\ast\psi)_2(x_2^\prime,x_1^\prime) \\
    &= \psi_2(f_1(x_1^\prime),f_1(x_2^\prime)) + \psi_2(f_1(x_2^\prime),f_1(x_1^\prime)) + \psi_1(f_2(x_1^\prime,x_2^\prime)) + \psi_1(f_2(x_2^\prime,x_1^\prime)) \\
    &= -\psi_1(h_2(f_1(x_1^\prime),f_1(x_2^\prime))) + \psi_1(f_2(x_1^\prime,x_2^\prime) + f_2(x_2^\prime,x_1^\prime)) \qquad\quad \text{by Equation~\eqref{Eq: f2}} \\
    &= -\psi_1(f_1(h_2^\prime(x_1^\prime,x_2^\prime))) = -(f^\ast\psi)_1(h_2^\prime(x_1^\prime,x_2^\prime)).
\end{align*}

For the second equation, we compute, on the one hand,
\begin{align*}
  (f^\ast D\omega)_0^{(0)}&(x^\prime_1,x^\prime_2,x^\prime_3) = (D\omega)_0^{(0)}(f_1(x_1^\prime),f_1(x_2^\prime),f_1(x_3^\prime)) \\
  &\quad +(D\omega)_1^{(0)}(f_1(x_3^\prime)\mid f_2(x_1^\prime,x_2^\prime))
  - (D\omega)_1^{(0)}(f_1(x_2^\prime)\mid f_2(x_1^\prime,x_3^\prime)) + (D\omega)_1^{(0)}(f_1(x_1^\prime)\mid f_2(x_2^\prime,x_3^\prime)),
\end{align*}
where
\begin{align*}
  (D\omega)_0^{(0)}&(f_1(x_1^\prime),f_1(x_2^\prime),f_1(x_3^\prime)) = -d\omega_0^{(-1)}(f_1(x_1^\prime),f_1(x_2^\prime),f_1(x_3^\prime)) \\
  &+ l_{f_1(x_1^\prime)} \omega_0^{(0)}(f_1(x_2^\prime),f_1(x_3^\prime)) - l_{f_1(x_2^\prime)} \omega_0^{(0)}(f_1(x_1^\prime),f_1(x_3^\prime)) + l_{f_1(x_3^\prime)} \omega_0^{(0)}(f_1(x_1^\prime),f_1(x_2^\prime)) \\
  &- \omega_0^{(0)}([f_1(x_1^\prime),f_1(x_2^\prime)],f_1(x_3^\prime)) - \omega_0^{(0)}(f_1(x_2^\prime),[f_1(x_1^\prime),f_1(x_3^\prime)]) + \omega_0^{(0)}(f_1(x_1^\prime),[f_1(x_2^\prime),f_1(x_3^\prime)]),
\end{align*}
and
\begin{align*}
  (D\omega)_1^{(0)}(f_1(x_3^\prime)\mid f_2(x_1^\prime,x_2^\prime)) &= -d\omega_1^{(-1)}(f_1(x_3^\prime)\mid f_2(x_1^\prime,x_2^\prime)) + \omega_0^{(0)}(df_2(x_1^\prime,x_2^\prime), f_1(x_3^\prime)) \\
  &\quad + l_{f_1(x_3^\prime)}\omega_1^{(0)}(f_2(x_1^\prime,x_2^\prime)) + \omega_1^{(0)}([f_2(x_1^\prime,x_2^\prime),f_1(x_3^\prime)]), \\
  (D\omega)_1^{(0)}(f_1(x_2^\prime)\mid f_2(x_1^\prime,x_3^\prime)) &= -d\omega_1^{(-1)}(f_1(x_2^\prime)\mid f_2(x_1^\prime,x_3^\prime)) + \omega_0^{(0)}(df_2(x_1^\prime,x_3^\prime), f_1(x_2^\prime)) \\
  &\quad + l_{f_1(x_2^\prime)}\omega_1^{(0)}(f_2(x_1^\prime,x_3^\prime)) + \omega_1^{(0)}([f_2(x_1^\prime,x_3^\prime),f_1(x_2^\prime)]),\\
  (D\omega)_1^{(0)}(f_1(x_1^\prime)\mid f_2(x_2^\prime,x_3^\prime)) &= -d\omega_1^{(-1)}(f_1(x_1^\prime)\mid f_2(x_2^\prime,x_3^\prime)) + \omega_0^{(0)}(df_2(x_2^\prime,x_3^\prime), f_1(x_1^\prime)) \\
  &\quad + l_{f_1(x_1^\prime)}\omega_1^{(0)}(f_2(x_2^\prime,x_3^\prime)) + \omega_1^{(0)}([f_2(x_2^\prime,x_3^\prime),f_1(x_1^\prime)]).
\end{align*}
On the other hand, we have
\begin{align*}
   (Df^\ast\omega)_0^{(0)}(x_1^\prime,x_2^\prime,x_3^\prime) &= (\delta f^\ast\omega)_0^{(0)}(x_1^\prime,x_2^\prime,x_3^\prime) + (d_{\LP}f^\ast\omega)_0^{(0)}(x_1^\prime,x_2^\prime,x_3^\prime),
\end{align*}
where
\begin{align*}
 (\delta f^\ast\omega)_0^{(0)}(x_1^\prime,x_2^\prime,x_3^\prime) &= -d\omega_0^{(-1)}(f_1(x_1^\prime),f_1(x_2^\prime),f_1(x_3^\prime)) - d\omega_1^{(-1)}(f_1(x_3^\prime)\mid f_2(x_1^\prime,x_2^\prime)) \\
 &+ d\omega_1^{(-1)}(f_1(x_2^\prime)\mid f_2(x_1^\prime,x_3^\prime)) - d\omega_1^{(-1)}(f_1(x_1^\prime)\mid f_2(x_2^\prime,x_3^\prime)),
\end{align*}
and
\begin{align*}
   (d_{\LP}f^\ast\omega)_0^{(0)}(x_1^\prime,x_2^\prime,x_3^\prime) &= l^\prime_{x_1^\prime}(f^\ast\omega)_0^{(0)}(x_2^\prime,x_3^\prime) - l^\prime_{x_2^\prime}(f^\ast\omega)_0^{(0)}(x_1^\prime,x_3^\prime) + l^\prime_{x_3^\prime}(f^\ast\omega)_0^{(0)}(x_1^\prime,x_2^\prime) \\
  &\quad - (f^\ast\omega)_0^{(0)}([x_1^\prime,x_2^\prime]^\prime,x_3^\prime) - (f^\ast\omega)_0^{(0)}(x_2^\prime,[x_1^\prime,x_3^\prime]^\prime) + (f^\ast\omega)_0^{(0)}(x_1^\prime,[x_2^\prime,x_3^\prime]^\prime) \\
  &= l^\prime_{x_1^\prime}\omega_0^{(0)}(f_1(x_2^\prime),f_1(x_3^\prime)) - l^\prime_{x_2^\prime}\omega_0^{(0)}(f_1(x_1^\prime),f_1(x_3^\prime)) + l^\prime_{x_3^\prime}\omega_0^{(0)}(f_1(x_1^\prime),f_1(x_2^\prime)) \\
  &\quad + l^\prime_{x_1^\prime}\omega_1^{(0)}(f_2(x_2^\prime,x_3^\prime)) - l^\prime_{x_2^\prime}\omega_1^{(0)}(f_2(x_1^\prime,x_3^\prime)) + l^\prime_{x_3^\prime}\omega_1^{(0)}(f_2(x_1^\prime,x_2^\prime)) \\
  &\quad - \omega_0^{(0)}(f_1([x_1^\prime,x_2^\prime]^\prime),f_1(x_3^\prime)) - \omega_0^{(0)}(f_1(x_2^\prime),f_1([x_1^\prime,x_3^\prime]^\prime)) + \omega_0^{(0)}(f_1(x_1^\prime),f_1([x_2^\prime,x_3^\prime]^\prime)) \\
  &\quad - \omega_1^{(0)}(f_2([x_1^\prime,x_2^\prime]^\prime,x_3^\prime)) - \omega_1^{(0)}(f_2(x_2^\prime,[x_1^\prime,x_3^\prime]^\prime)) + \omega_1^{(0)}(f_2(x_1^\prime,[x_2^\prime,x_3^\prime]^\prime)).
\end{align*}
Then it follows from Equations~\eqref{Eq: f2ischainhomotopy},~\eqref{Eq: f2andbrackets},~\eqref{Eq: pullback rep}, and a direct verification that
\begin{align*}
 (f^\ast D\omega)_0^{(0)}(x^\prime_1,x^\prime_2,x^\prime_3) &= (Df^\ast\omega)_0^{(0)}(x_1^\prime,x_2^\prime,x_3^\prime).
\end{align*}

\section{Cohomology of injective hemistrict Lie 2-algebras}\label{Sec: app}
A hemistrict Lie 2-algebra $L = (\vec{L},[-,-],h_2)$ is said to be injective, if the differential $d$ of the underlying 2-term complex $\vec{L} = L^{-1} \xrightarrow{d} L^0$ is injective, i.e., there is a short exact sequence of graded vector spaces
\begin{equation}\label{SES of injective hemistrict}
  0 \rightarrow L^{-1} \xrightarrow{d} L^0 \xrightarrow{\pr} L^0/dL^{-1} \rightarrow 0.
\end{equation}
Let $L = (\vec{L},[-,-],h_2)$ be an injective hemistrict Lie 2-algebra.
The vector space $H^\bullet(\vec{L}) = L^0/dL^{-1}$, together with the bracket $\{-,-\}$ defined by
\[
  \{\bar{x},\bar{y}\} = \pr([x,y]), \qquad \forall \bar{x},\bar{y} \in L^0/dL^{-1},
\]
for all $x,y \in L^0$ such that $\pr(x) = \bar{x}, \pr(y) = \bar{y}$, is a Lie algebra, which will be denoted by $L_{\Lie}$.
Meanwhile, according to Roytenberg~\cite{Roy1}, the skew-symmetrization on the bracket $[-,-]$ gives rise to a semistrict Lie 2-algebra $\tilde{L} = (\vec{L},\tilde{l}_2, \tilde{l}_3)$, where
\begin{align}\label{Eq: ltilde2}
  \tilde{l}_2(x,y) &= \frac{1}{2}\left([x,y] - (-1)^{\abs{x}\abs{y}}[y,x]\right), \;\;\forall x,y \in L^{-1} \oplus L^0, \\ \label{Eq: ltilde3}
  \tilde{l}_3(x,y,z) &= -\frac{1}{6}\left(h_2(\tilde{l}_2(x,y),z) + h_2(\tilde{l}_2(y,z), x) + h_2(\tilde{l}_2(z,x),y)\right),\;\;\forall x,y,z \in L^0.
\end{align}
The main purpose of this section is to build isomorphisms of these three objects on the cohomology level.
\subsection{Main theorem}
Assume that $V = (\vec{V},l,r,h_v)$ is a representation of $L$ such that $l_\alpha = 0$ for all $\alpha \in L^{-1}$. It follows that $l$ induces a representation of the Lie algebra $L_{\Lie}$ as well as a representation of the semistrict Lie 2-algebra $\tilde{L}$ on $\vec{V}$, i.e., $V$ is both an $L_{\Lie}$-module and an $\tilde{L}$-module.

Here is our main theorem:
\begin{thm}\label{main theorem}
   Let $L$ be an injective hemistrict Lie 2-algebra and $V = (\vec{V},l,r,h_v)$ a representation such that $l_\alpha = 0$ for all $\alpha \in L^{-1}$.
   \begin{compactenum}
   \item Assume that the alternator $h_2$ satisfies $h_2(d\alpha,d\beta) = 0$ for any $\alpha,\beta \in L^{-1}$. Then the cohomology of the representation $V$ of $L$ is isomorphic to the Chevalley-Eilenberg cohomology of the Lie algebra $L_{\Lie}$ with coefficient $V$, i.e.,
     \[
       H^\bullet(L,V) \cong H^\bullet_{\CE}(L_{\Lie},V).
     \]
   \item The cohomology of the semistrict Lie 2-algebra $\tilde{L}$ obtained from skew-symmetrization is isomorphic to the Chevalley-Eilenberg cohomology of the Lie algebra $L_{\Lie}$, i.e.,
       \[
       H^\bullet(\tilde{L},V) \cong H^\bullet_{\CE}(L_{\Lie},V).
       \]
   \end{compactenum}
\end{thm}

As a consequence, we have
\begin{Cor}\label{Cor of L et tildeL}
  Under the assumptions as in Theorem~\ref{main theorem}, we have
  \[
   H^\bullet(L,V) \cong H^\bullet(\tilde{L},V).
  \]
\end{Cor}

\begin{Rem}
  In fact, each representation $V$ of a hemistrict Lie 2-algebra $L = (\vec{L},[-,-],h_2)$ induces a representation $\{\mu_k\}_{k=1}^3$ of the semistrict Lie 2-algebra $\tilde{L}$ obtained from $L$ via skew-symmetrization on $\vec{V}$, where
  \begin{align*}
    \mu_1 &= d\colon V^{-1} \to V^0, & \mu_2 &= l\colon L^\bullet \otimes V^\bullet \to V^\bullet,
  \end{align*}
  and $\mu_3\colon L^0 \wedge L^0 \otimes V^0 \to V^{-1}$ is specified by
  \[
   \mu_3(x,y,v) = \frac{1}{2}l(h_2(x,y),v),
  \]
  for all $x,y \in L^0$ and $v \in V^0$. It is natural to ask if the isomorphism in Corollary~\ref{Cor of L et tildeL} holds for general representations. We will investigate this problem in a incoming paper.
\end{Rem}

\subsection{Proof of Theorem~\ref{main theorem}}
\subsubsection{Proof of the first statement}
Note that the associate Lie algebra $L_{\Lie}$ may also be viewed as a hemistrict Lie 2-algebra $(0 \oplus L^0/dL^{-1},[-,-], h_2=0)$. The projection $\pr\colon L^0 \to L^0/dL^{-1}$ extends to a morphism of hemistrict Lie 2-algebras
\[
 f:= (\pr,0)\colon L \longrightarrow L_{\Lie}.
\]

The first observation is the following
\begin{lem}\label{Lemma: FG}
Let $j\colon L^0/dL^{-1} \rightarrow L^0$ be a splitting of~\eqref{SES of injective hemistrict}. There associates a morphism of hemistrict Lie 2-algebras
\[
 g := (g_1,g_2) \colon L_{\Lie} \longrightarrow L,
\]
defined by
\begin{align*}
  g_1(\bar{x}) &= j(\bar{x}), & g_2(\bar{x},\bar{y}) &= \pr_{-1}([j(\bar{x}),j(\bar{y})]),
\end{align*}
for all $\bar{x},\bar{y} \in L^0/dL^{-1}$, where $\pr_{-1}\colon L^0 \rightarrow L^{-1}$ is the projection specified by $\id_L = j \circ \pr + d \circ \pr_{-1}$, such that
\begin{align*}
 f \circ g &= \id \colon L_{\Lie} \to L_{\Lie}.
\end{align*}
Furthermore, the map $\theta:=\pr_{-1}\colon L^0 \to L^{-1}$ gives rise to a 2-morphism $\theta\colon \id_L \Rightarrow g \circ f$.
\end{lem}
\begin{proof}
  We first show that $g$ is well-defined by verifying Equations~\eqref{Eq: f1uptof2},~\eqref{Eq: h2andf2} and~\eqref{Eq: Jacobiandf2} in this case. In fact, we have, by definition,
  \begin{align*}
    [g_1(\bar{x}),g_1(\bar{x})] - g_1(\{\bar{x},\bar{y}\}) &= [j(\bar{x}),j(\bar{y})] - j(\pr([j(\bar{x}),j(\bar{y})])) \\
    &= d\pr_{-1}([j(\bar{x}),j(\bar{y})]) = dg_2(\bar{x},\bar{y}),
  \end{align*}
  \begin{align*}
    g_2(\bar{x},\bar{y}) + g_2(\bar{y},\bar{x}) &= \pr_{-1}([j(\bar{x}),j(\bar{y})]+[j(\bar{y}),j(\bar{x})]) = \pr_{-1}(dh_2(j(\bar{x}),j(\bar{y}))) \\
    &= h_2(j(\bar{x}),j(\bar{y})) = h_2(g_1(\bar{x}),g_1(\bar{y})),
  \end{align*}
  and
  \begin{align*}
    &\quad [g_1(\bar{x}),g_2(\bar{y},\bar{z})] - [g_1(\bar{y}),g_2(\bar{x},\bar{z})] - [g_2(\bar{x},\bar{y}),\bar{z}] - g_2(\{\bar{x},\bar{y}\},\bar{z}) - g_2(\bar{y},\{\bar{x},\bar{z}\}) + g_2(\bar{x},\{\bar{y},\bar{z}\}) \\
    &= \pr_{-1}([j(\bar{x}),[j(\bar{y}),j(\bar{z})]] - [j(\bar{y}),[j(\bar{x}),j(\bar{z})]] - [[j(\bar{x}),j(\bar{y})],j(\bar{z})]) = 0.
  \end{align*}
Meanwhile, since
\begin{align*}
  (f\circ g)_1 &= f_1 \circ g_1 = \pr \circ j = \id, & (f\circ g)_2 &= f_2 \circ (g_1 \otimes g_1) + f_1 \circ g_2 = 0,
\end{align*}
it follows that $f \circ g = \id \colon L_{\Lie} \to L_{\Lie}$.

Finally, since on the one hand
\begin{align*}
  (g \circ f)_1(x) &= g_1(f_1(x)) = j(\pr(x)) = x - d\pr_{-1}(x) = x -d(\theta)(x),\quad \forall x, y \in L^{\bullet},
\end{align*}
thus $\theta$ is a chain homotopy from $\id_L$ to $g \circ f$. On the other hand, note that
\begin{align*}
  (g\circ f)_2(x,y) &= g_2(f_1(x),f_1(y)) + g_1(f_2(x,y)) = \pr_{-1}([j(\pr(x)),j(\pr(y))]) = \pr_{-1}([x-d\theta(x),y-d\theta(y)]) \\
  &= \theta([x,y]) - \pr_{-1}(d[x,\theta(y)]) - \pr_{-1}(d[\theta(x),y]) + \pr_{-1}(d[d\theta(x),\theta(y)]) \\
  &= \theta([x,y]) - [x,\theta(y)] - [\theta(x),y] + [d\theta(x),\theta(y)] \\
  &= \theta([x,y]) - [(g\circ f)_1(x),\theta(y)] - [\theta(x),y],
\end{align*}
for all $x,y\in L^0$. Thus, $\theta \colon \id_L \Rightarrow g\circ f$ is a 2-morphism.
\end{proof}

Let $V = (\vec{V},l,r,h_v)$ be a representation of $L$ such that $l_\alpha = 0$ for all $\alpha \in L^{-1}$. By Lemma~\ref{Lemma: FG}, $V$ is $(g\circ f)$-compatible, and the pullback representation by $g \circ f$ coincides with $V$. By Proposition~\ref{prop: pullback}, we have a cochain map
\[
 (g\circ f)^\ast\colon C^\bullet(L,V) \to C^\bullet(L,V).
\]
Moreover, we have the following
\begin{lem}\label{Lemma: Theta}
  Under the assumption that $h_2(d\alpha,d\beta) = 0$ for any $\alpha,\beta \in L^{-1}$, the 2-morphism $\theta\colon \id_L \Rightarrow g \circ f$ gives rise to a 2-morphism $\Theta\colon \id_{C^\bullet(L,V)} \Rightarrow (g\circ f)^\ast$ in the 2-category of cochain complexes, i.e., there exists a chain homotopy $\Theta\colon C^\bullet(L,V) \to C^{\bullet-1}(L,V)$ such that for all $\omega \in C^\bullet(L,V)$,
  \[
   \omega - (g\circ f)^\ast(\omega) = D\Theta(\omega) + \Theta(D\omega).
  \]
\end{lem}
\begin{proof}
For each $\omega \in C^\bullet(L,V)$, define
\begin{align*}
  &\quad (\Theta\omega)_k^{(p)}(x_1,\cdots,x_{n-2k-p-1} \mid \alpha_1,\cdots,\alpha_k) \\
  &= \sum_{q=0}^{\lfloor\frac{n-p}{2}\rfloor -k-1} \sum_{\substack{i_1<j_1;\cdots;i_q<j_q, \\ i_1<\cdots<i_q}} \sum_{i \notin \{i_1,j_1,\cdots,i_q,j_q\}} \sum_{\substack{m_1<\cdots<m_t;l_1<\cdots<l_r \\ \{i,i_s,j_s,m,l\} = \{1,\cdots,n-2k-p-1\}}}\\
  &\qquad\qquad\quad  \frac{(-1)^{i-1}q!}{(k+t+1)\cdots(k+t+q+1)}  \sgn(i_1<j_1;\cdots;i_q<j_q) \prod_{s=1}^q\sgn(i_s-i)\sgn(j_s-i) \\
  &\qquad\qquad\qquad\qquad  \omega^{(p)}_{k+q+1}((g\circ f)_1(x_{l_1}),\cdots,d\theta(x_{m_1}),\cdots,(g\circ f)_1(x_{l_r}),\cdots, d\theta(x_{m_t}) \\
  &\qquad\qquad\qquad\qquad\qquad\qquad\qquad \mid \theta(x_i), (g\circ f)_2(x_{i_1},x_{j_1}), \cdots, (g\circ f)_2(x_{i_q}, x_{j_q}),\alpha_1, \cdots, \alpha_k),
\end{align*}
for all $x_i \in L^{0}, \alpha_j \in L^{-1}$, where $\sgn(i_1<j_1;\cdots;i_q<j_q)$ is defined in Equation~\eqref{Eq: sgn}, and $\sgn(i_s-i)$ equals either $1$ if $i_s-i>0$ or $-1$ otherwise.
Here the $(r+t)$-terms $\{(g\circ f)_1(x_{l_1}),\cdots,d\theta(x_{m_1}),\cdots, (g\circ f)_1(x_{l_r}),\cdots, d\theta(x_{m_t}\}$ are placed by the monotone increasing order of the indices $l_1,\cdots,m_t$.
The result then follows from a straightforward verification.
\end{proof}

By Lemma~\ref{Lemma: FG}, $V$ can be pulled back by $g$ to an $L_{\Lie}$-module, which will also be denoted by $V$. Meanwhile, it is also clear that the $L$-module coincides with the pullback of the $L_{\Lie}$-module $V$ by $f$.
Note that the cohomology of the hemistrict Lie 2-algebra $L_{\Lie}$ is exactly the Chevalley-Eilenberg cohomology of the Lie algebra $L_{\Lie}$. Hence, by Proposition~\ref{prop: pullback}, Lemma~\ref{Lemma: FG}, and Lemma~\ref{Lemma: Theta}, we have
\[
 H^\bullet(L,V) \cong H^\bullet_{\CE}(L_{\Lie},V).
\]

\subsubsection{Proof of the second statement}
To prove the second statement of Theorem~\ref{main theorem}, we need the homological perturbation lemma (cf.~\cite{BCSX}), which we recall as follows:

Let us start with a homotopy contraction of cochain complexes:
\[
 \begin{tikzcd}
  (A, d_A) \arrow[loop left, distance=2em, start anchor={[yshift=-1ex]west}, end anchor={[yshift=1ex]west}]{}{h} \arrow[r,yshift = 0.7ex, "\phi"] & (B, d_B) \arrow[l,yshift = -0.7ex, "\psi"],
  \end{tikzcd}
\]
where both $\psi$ and $\phi$ are maps of cochain complexes, and $h$ is the  degree $(-1)$ chain homotopy, satisfying the following two equations
\begin{align*}
  \phi \circ \psi &= \id_B, & \psi \circ \phi &= \id_A - [d_A,h],
\end{align*}
together with the side conditions
\begin{align*}
  h \circ \psi &= 0, & \phi \circ h &= 0, & h^2 &= 0.
\end{align*}

\begin{lem}[The Perturbation Lemma]\label{HPT}
 Let $(A,D_A = d_A + \rho)$ be a perturbation of $(A,d_A)$. Then we have a new homotopy contraction
 \[
 \begin{tikzcd}
  (A, D_A) \arrow[loop left, distance=2em, start anchor={[yshift=-1ex]west}, end anchor={[yshift=1ex]west}]{}{H} \arrow[r,yshift = 0.7ex, "\Phi"] & (B, D_B) \arrow[l,yshift = -0.7ex, "\Psi"],
  \end{tikzcd}
\]
 where
 \begin{align*}
   D_B &= d_B + \sum_{k\geq 0} \phi(\rho h)^k \rho \psi, &
   \Phi &= \sum_{k\geq 0}\phi (\rho h)^k, \\
   H &= \sum_{k\geq 0} (h\rho)^kh = \sum_{k\geq 0} h(\rho h)^k, & \Psi &= \sum_{k \geq 0} (h\rho)^k\psi.
 \end{align*}
\end{lem}

Now we analyze the two cochain complexes in our situation: We will denote by $K \subset L^0$ the image of $d \colon L^{-1} \to L^0$ in the sequel.

On the one hand, the Chevalley-Eilenberg cochain complex of the Lie algebra $L_{\Lie}$ is
\[
 B := \oplus_{n\geq0}C^n(L_{\Lie},V) = \oplus_{n\geq0}\oplus_{p+r = n}\wedge^p (L^0/K)^\vee \otimes V^r,
\]
with the differential $D_B = d_V + d_{\CE}$, where
\begin{compactenum}
  \item $d_V$ comes from the differential of the 2-term complex $\vec{V}$, which increases the index $r$ by $1$, i.e.,
    \[
     d_V\colon \wedge^p(L^0/K)^\vee \otimes V^r \rightarrow \wedge^p(L^0/K)^\vee \otimes V^{r+1},
    \]
    where $V^r = 0$ if $r \neq -1,0$;
  \item $d_{\CE}$ is the Chevalley-Eilenberg differential of the Lie algebra $L_{\Lie}$ with coefficient $V^\bullet = V^{-1} \oplus V^0$, which increases the index $p$ by $1$, i.e.,
    \[
     d_{\CE}\colon \wedge^p(L^0/K)^\vee \otimes V^r \rightarrow \wedge^{p+1}(L^0/K)^\vee \otimes V^{r}.
    \]
\end{compactenum}

On the other hand, the cochain complex of the semistrict Lie 2-algebra $\tilde{L}$ with coefficient $V$ is, by definition,
\[
A := C^\bullet(\tilde{L},V) = \oplus_{n\geq0} C^n(\tilde{L},V) = \oplus_{n\geq0}\oplus_{p+2q+r = n}\wedge^p(L^0)^\vee \otimes S^q((L^{-1}[1])^\vee) \otimes V^r,
\]
with the differential $D_A = d_V + \delta + \rho$, where
\begin{compactenum}
\item $d_V$ also comes from the differential of the 2-term complex $\vec{V}$, which increases the index $r$ by $1$, i.e.,
    \[
     d_V\colon \wedge^p(L^0)^\vee \otimes S^q((L^{-1}[1])^\vee) \otimes V^r \rightarrow \wedge^p(L^0)^\vee \otimes S^q((L^{-1}[1])^\vee) \otimes V^{r+1};
    \]
\item $\delta$ is induced from the differential of the 2-term complex $\vec{L}$, which decreases the index $p$ by $1$ and increases the index $q$ by $1$ at the same time, i.e.,
\[
  \delta\colon \bigoplus_{p,q \geq 0}\wedge^{p+1}(L^0)^{\vee} \otimes S^q((L^{-1}[1])^\vee) \otimes V^r \rightarrow \bigoplus_{p,q \geq 0}\wedge^{p}(L^0)^{\vee} \otimes S^{q+1}((L^{-1}[1])^\vee) \otimes V^r;
\]
\item $\rho = -\tilde{l}_2^\vee + \tilde{l}_3^\vee + l^\vee$ is the sum of duals of the 2-bracket $\tilde{l}_2$, the 3-bracket $\tilde{l}_3$, defined by Equation~\eqref{Eq: ltilde2},~\eqref{Eq: ltilde3}, and the left action $l$ of $L$ on $V^\bullet$.
\end{compactenum}
It is clear that $(d_V + \delta)^2 = 0$. Thus, the cochain complex $A$ results from a perturbation of the complex
\[
 A^\prime := (\oplus_{n\geq0} C^n(\tilde{L},V), d_V + \delta).
\]

Let us choose a splitting $j: L^0/K \rightarrow L^0$ of the following exact sequence of vector spaces
\[
  0 \rightarrow K \xrightarrow{i} L^0 \xrightarrow{\pr} L^0/K \rightarrow 0.
\]
Thus, $L^0 \cong K \oplus L^0/K$ and $\wedge^p(L^0)^\vee \cong  \oplus_{t+s=p}\wedge^t (L^0/K)^\vee \otimes \wedge^s K^\vee$.

The first observation is the following
\begin{lem}\label{lemma}
  There is a homotopy contraction
\[
 \begin{tikzcd}
  A^\prime := (\oplus_{n\geq0} C^n(\tilde{L},V), d_V + \delta) \arrow[loop left, distance=2em, start anchor={[yshift=-1ex]west}, end anchor={[yshift=1ex]west}]{}{h} \arrow[r,yshift = 0.7ex, "\phi"] &  (\oplus_{n\geq0} \oplus_{p+r=n} \wedge^p (L^0/K)^\vee \otimes V^r, d_V) := B^\prime \arrow[l,yshift = -0.7ex, "\psi"],
  \end{tikzcd}
\]
where
  \begin{align*}
    \psi\colon & \oplus_{p+r=n}\wedge^p (L^0/K)^\vee \otimes V^r \xrightarrow{\pr^\vee \otimes \id_{V^r}} \oplus_{p+r=n}\wedge^p(L^0)^\vee \otimes V^r \hookrightarrow C^n(\tilde{L},V), \\
    \phi\colon & C^n(\tilde{L},V) \twoheadrightarrow \oplus_{p+r=n}\wedge^p (L^0)^\vee \otimes V^r \xrightarrow{j^\vee \otimes \id_{V^r}} \oplus_{p+r=n}\wedge^p (L^0/K)^\vee \otimes V^r,
  \end{align*}
 and the degree $(-1)$ chain homotopy
 \[
  h\colon \wedge^t (L^0/K)^\vee \otimes \wedge^s K^\vee \otimes  S^{q+1}((L^{-1}[1])^\vee) \otimes V^r \rightarrow \wedge^t (L^0/K)^\vee \otimes \wedge^{s+1} K^\vee \otimes S^q((L^{-1}[1])^\vee) \otimes V^r
 \]
 is specified by for all
 \[
 \omega \in \Hom(\wedge^t L^0/K \otimes \wedge^s K \otimes  S^{q+1}(L^{-1}[1]), V^r) \cong \wedge^t (L^0/K)^\vee \otimes \wedge^s K^\vee \otimes  S^{q+1}((L^{-1}[1])^\vee) \otimes V^r,
 \]
 \begin{align*}
  &\quad h(\omega)(\bar{x}_1,\cdots,\bar{x}_t, k_1,\cdots,k_{s+1} \mid \alpha_1,\cdots,\alpha_q) \\
  &= \begin{cases} \frac{1}{s+q}\sum_{j=1}^{s+1}(-1)^{s+1-j}\omega(\bar{x}_1,\cdots,\bar{x}_t, \cdots,\widehat{k_j},\cdots \mid \pr_{-1}(k_j),\alpha_1,\cdots,\alpha_q), & \text{if $s+q > 0$}, \\
  0,& \text{otherwise},
  \end{cases}
 \end{align*}
 for all $\bar{x}_1,\cdots,\bar{x}_t \in L^0/K, k_1,\cdots,k_{s+1} \in K, \alpha_1,\cdots,\alpha_q \in L^{-1}$. Here $\pr_{-1}: K \to L^{-1}$ is the inverse of the isomorphism $d: L^{-1} \to K \subset L^0$.
\end{lem}
The proof of this lemma is straightforward, thus omitted.

Applying the perturbation Lemma~\ref{HPT} to the contraction in Lemma~\ref{lemma}, we prove the following
\begin{lem}\label{main thm}
  There is a homotopy contraction
 \[
 \begin{tikzcd}
  A = (\oplus_{n\geq0} C^n(\tilde{L},V), D_A = d_V + \delta + \rho) \arrow[loop left, distance=2em, start anchor={[yshift=-1ex]west}, end anchor={[yshift=1ex]west}]{}{H} \arrow[r,yshift = 0.7ex, "\Phi"] & B = (\oplus_{n\geq0}\wedge^n (L^0/K)^\vee \otimes V, D_B:= d_V + d_{\CE}) \arrow[l,yshift = -0.7ex, "\Psi"].
  \end{tikzcd}
\]
\end{lem}
\begin{proof}
It suffices to show that
\[
 \sum_{k \geq 0}\phi (\rho h)^k \rho\psi = d_{\CE}: \wedge^p (L^0/K)^\vee \otimes V^r \rightarrow \wedge^{p+1} (L^0/K)^\vee \otimes V^r,
\]
where $\phi,\psi$ and $h$ are defined in Lemma~\ref{lemma} and $\rho = \tilde{l}_2^\vee + \tilde{l}_3^\vee + l^\vee$.

In fact, since
\[
 \tilde{l}_2^\vee \circ \psi\colon \wedge^p (L^0/K)^\vee \otimes V^r \rightarrow \wedge^{p+1} (L^0)^\vee \otimes V^r,
\]
by the definition of $\tilde{l}_2$, and
\[
 l^\vee\colon \wedge^p (L^0/K)^\vee \otimes V^r \rightarrow \wedge^{p+1} (L^0)^\vee \otimes V^r,
\]
by the assumption that $\iota_\alpha l^\vee = l_\alpha = 0$ for all $\alpha \in L^{-1}$, it follows that
\[
 h \rho \psi = h(-\tilde{l}_2^\vee + \tilde{l}_3^\vee + l^\vee) \psi = h(\tilde{l}_2^\vee + l^\vee) \psi = 0: \wedge^p (L^0/K)^\vee \otimes V^r \rightarrow C^{p+r}(\tilde{L},V).
\]
Thus,
\[
 \sum_{k \geq 0}\phi (\rho h)^k \rho\psi = \phi\rho\psi = \phi(-\tilde{l}_2^\vee + \tilde{l}_3^\vee + l^\vee)\psi = \phi (-\tilde{l}_2^\vee + l^\vee) \psi = d_{\CE}: \wedge^p (L^0/K)^\vee \otimes V^r \rightarrow \wedge^{p+1} (L^0/K)^\vee \otimes V^r.
\]
To see the reason why the last equality holds, it suffices to prove the case $p=1$: We compute for each $\xi \in \Hom(L^0/K,V^r)$, $x, y \in L^0$ such that $\bar{x} = \pr(x), \bar{y} = \pr(y) \in L^0/K$,
\begin{align*}
 \phi(\tilde{l}_2^\vee(\psi(\xi)))(\bar{x},\bar{y}) &= \xi(\pr(\tilde{l}_2(j(\bar{x}),j(\bar{y}))))
  = \frac{1}{2}\xi(\pr([j(\bar{x}),j(\bar{y})] - [j(\bar{y}),j(\bar{x})])) \\
 &= \frac{1}{2}\xi(\pr([x,y] - [y,x])) = \frac{1}{2}\xi(\pr([x,y] - [y,x]) + \pr(dh_2(x,y))) \\
 &= \frac{1}{2}\xi(\pr([x,y] - [y,x]) + \pr([x,y] + [y,x])) = \xi(\{\bar{x},\bar{y}\}),
\end{align*}
and
\begin{align*}
  \phi(l^\vee(\psi(\xi)))(\bar{x},\bar{y}) &= l_{j(\bar{x})}\xi(\bar{y}) - l_{j(\bar{y})}\xi(\bar{x}),
\end{align*}
which implies that $\phi(l^\vee(\psi(\xi))) - \phi(\tilde{l}_2^\vee(\psi(\xi))) = d_{\CE}(\xi)$ as desired.
\end{proof}
As an immediate consequence, we have $H^\bullet_{\CE}(\tilde{L},V) \cong H^\bullet_{\CE}(L_{\Lie},V)$, which completes the proof of Theorem~\ref{main theorem}.

\subsection{Application: Leibniz algebras}
Let $(\g,[-,-]_\g)$ be a Leibniz algebra with Leibniz kernel $K$. By Example~\ref{Ex:Leibniz algebra}, we have an injective hemistrict Lie 2-algebra
\[
 L_\g := (K[1] \hookrightarrow \g,[-,-]_\g,h_2).
\]
The associated Lie algebra is commonly denoted by $\g_{\Lie}$. It is clear that $L_\g$ satisfies the assumptions in Theorem~\ref{main theorem}. Note also that the alternator $h_2$ is surjective in this case. In fact, any injective hemistrict Lie $2$-algebra with surjective alternator $h_2$ is of this form (cf.~\cite{Roy1}).

Meanwhile, according to Roytenberg~\cite{Roy1}, Sheng and Liu~\cite{SL}, the skew-symmetrization of the Leibniz bracket $[-,-]_\g$ gives rise to a semistrict Lie 2-algebra
\[
 \mathcal{G} := (K[1] \hookrightarrow \g, \widetilde{l}_2,\widetilde{l}_3),
\]
where $\widetilde{l}_2$ is the skew-symmetrization of $[-,-]_\g$, i.e,
\begin{align*}
  \widetilde{l}_2(x,y) &= \frac{1}{2}([x,y]_\g - [y,x]_\g), & \widetilde{l}_2(x,\alpha) &= -\widetilde{l}_2(\alpha,x) = \frac{1}{2}[x,\alpha]_\g,
\end{align*}
for all $x,y \in \g, \alpha \in K[1]$, and $\widetilde{l}_3: \wedge^3\g \to K[1]$ is defined by
\begin{align*}
 \widetilde{l}_3(x,y,z) &=-\frac{1}{12}\left(h_2([x,y]_\g-[y,x]_\g,z) + h_2([y,z]_\g-[z,y]_\g,x) + h_2([z,x]_\g-[x,z]_\g,y)\right) \\
  &= \frac{1}{4}\left([[z,y]_\g,x]_\g + [[x,z]_\g,y]_\g + [[y,x]_\g,z]_\g\right).
\end{align*}
Applying Theorem~\ref{main theorem}, we have the following
\begin{thm}\label{Thm: Leibniz algebra}
  Let $(\g,[-,-]_\g)$ be a Leibniz algebra with Leibniz kernel $K$ and $V$ a representation of $L_\g$ such that $l_\alpha = 0$ for all $\alpha \in K$. Then
  \[
   H^\bullet(L_\g,V) \cong H^\bullet(\mathcal{G},V) \cong H^\bullet_{\CE}(\g_{\Lie},V).
  \]
\end{thm}
Recall that the Leibniz kernel $K$ is a subset of the left center of $\g$. Thus the adjoint representation $\Ad_{L_\g}$ of $L_\g$, which arises from the adjoint representation of the Leibniz algebra $\g$ as in Example~\ref{Ex: AdLg}, satisfies the assumption in the above theorem. As a consequence,
\begin{Cor}
  Let $(\g,[-,-]_\g)$ be a Leibniz algebra. Then
  \[
   H^\bullet(L_\g, \Ad_{L_\g}) \cong H^\bullet(\mathcal{G},\Ad_{L_\g}) \cong H^\bullet_{\CE}(\g_{\Lie},\Ad_{L_\g}).
  \]
\end{Cor}

\end{document}